\documentclass[a4paper]{amsart}

\usepackage{amsthm}
\usepackage{amsmath}
\usepackage{amssymb}
\usepackage[dvips]{graphicx}
\usepackage{longtable}

\usepackage{bm}
\usepackage{braket}
\usepackage{setspace}
\usepackage{mathrsfs}
\usepackage{subfig}
\usepackage[all]{xy}
\usepackage{here}
\usepackage[toc,page]{appendix}
\usepackage{caption}

\usepackage{enumitem}
\setlist[enumerate]{label=(\roman*),font=\upshape}

\makeatletter
\def\mojiparline#1{
    \newcounter{mpl}
    \setcounter{mpl}{#1}
    \@tempdima=\linewidth
    \advance\@tempdima by-\value{mpl}zw
    \addtocounter{mpl}{-1}
    \divide\@tempdima by \value{mpl}
    \advance\kanjiskip by\@tempdima
    \advance\parindent by\@tempdima
}

\makeatother

\makeatother

\title[On a classification of irreducible periodic diffeomorphisms]{On a classification of irreducible periodic diffeomorphisms on surfaces which commute with certain involution}

\author{Norihisa Takahashi}
\address{Department of Mathematical Sciences,\\ Colleges of Science and Engineering,\\ Ritsumeikan University,\\ Nojihigashi 1-1-1, Kusatsu, Shiga, 525-8577, Japan}
\email{ntakaha@fc.ritsumei.ac.jp}

\author{Hiraku Nozawa}
\address{Department of Mathematical Sciences,\\ Colleges of Science and Engineering,\\ Ritsumeikan University,\\ Nojihigashi 1-1-1, Kusatsu, Shiga, 525-8577, Japan}
\thanks{The second author is supported by JSPS KAKENHI Grant Numbers 17K14195 and 20K03620.}
\email{hnozawa@fc.ritsumei.ac.jp}

\makeatletter
\@namedef{subjclassname@2020}{\textup{2020} Mathematics Subject Classification}
\makeatother

\date{}

\newtheorem{thm}{Theorem}[section]
\newtheorem{cor}[thm]{Corollary}
\newtheorem{prop}[thm]{Proposition}
\newtheorem*{thm:}{Main Theorem}
\newtheorem{lem}[thm]{Lemma}
\theoremstyle{definition}

\newtheorem{ex}[thm]{Example}
\theoremstyle{remark}
\newtheorem{rem}[thm]{Remark}

\newcommand{\bb}{\mathbb}

\newcommand{\opn}{\operatorname}
\newcommand{\ZZ}{\mathbb{Z}}
 
\usepackage{color}
\definecolor{darkgreen}{cmyk}{1,0,1,.2}
\definecolor{darkorchid}{rgb}{0.6, 0.2, 0.8}
\definecolor{persimmon}{rgb}{0.93, 0.35, 0.0}

\newdimen\theight
\def\TeXref#1{%
             \leavevmode\vadjust{\setbox0=\hbox{{\tt
                     \quad\quad  {\small \textrm #1}}}%
             \theight=\ht0
             \advance\theight by \lineskip
             \kern -\theight \vbox to
             \theight{\rightline{\rlap{\box0}}%
             \vss}%
             }}%

\begin{document}
\begin{abstract}
    Ishizaka classified up to conjugacy hyperelliptic periodic automorphisms of a surface.
    Here, an involution $I$ on a surface $\Sigma_{g}$ is hyperelliptic if and only if $\Sigma_{g}/\langle I \rangle$ is homeomorphic to $S^2$.
    In this article, we give a classification up to conjugacy for irreducible periodic automorphisms of a surface $\Sigma_{g}$ which commute with involutions $\iota$ such that $\Sigma_{g}/\langle \iota \rangle$ is homeomorphic to $T^{2}$.
\end{abstract}

\maketitle

\section{Introduction}

    Let $\Sigma_{g}$ of genus $g>1$ be a connected oriented closed surface.
    The \emph{mapping class group} $\operatorname{Mod}(\Sigma_g)$ of $\Sigma_g$ is the group of isotopy classes of orientation-preserving diffeomorphisms of $\Sigma_g$.
    We call elements of $\operatorname{Mod}(\Sigma_g)$ \emph{automorphisms} of $\Sigma_g$.
    An involution on $\Sigma_g$ is called \emph{hyperelliptic} if it fixes $2g+2$ points.
    It is well known that such an involution is unique up to conjugacy and has the maximum number of fixed points among involutions on $\Sigma_{g}$.
    An automorphism of $\Sigma_{g}$ which commutes with a hyperelliptic involution is called \emph{hyperelliptic}.
    Ishizaka \cite{Ishizaka3,Ishizaka} classified hyperelliptic periodic automorphisms up to conjugacy and gave right-handed Dehn twist presentations based on resolutions of singularities of families of Riemann surfaces.
    An involution $I$ of $\Sigma_{g}$ is hyperelliptic if and only if the quotient space $\Sigma_{g}/\langle I \rangle$ is homeomorphic to a sphere.

    In this paper, we will study periodic automorphisms of $\Sigma_{g}$ which commute with an involution $\iota_{g}$ whose quotient space $\Sigma_{g}/\langle \iota_{g} \rangle$ is homeomorphic to a torus (see Figure \ref{fig:iotag}).
        \begin{figure}[ht]
	    	\centering
		    \includegraphics[scale=0.45]{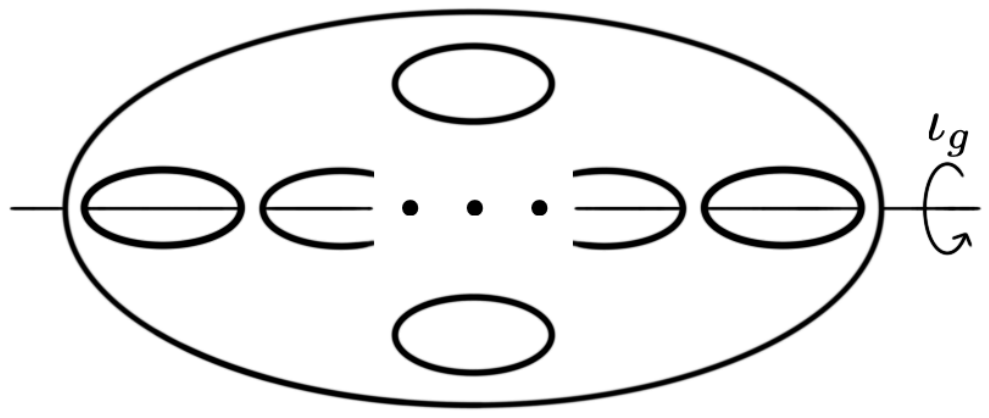}
    		\caption{An involution $\iota_{g}$ such that $\Sigma_g / \langle \iota_g \rangle$ is homeomorphic to $T^2$}
	    	\label{fig:iotag}
        \end{figure}
    There is a significant difference in the classification from the hyperelliptic case considered in this case in \cite{Ishizaka}.
    Recall that an automorphism $\varphi$ of $\Sigma_g$ is called \emph{reducible} if there is a nonempty set $\{ C_1,C_2,\dots, C_n \}$ of isotopy classes of essential simple closed curves in $\Sigma_g$ so that the geometrical intersection number of $C_i$ and $C_j$ is equal to $0$ for all $i, j \in \{ 1,2,\dots,n \}$ and so that $\{ \varphi(C_i) \mid i=1,2,\dots, n \} = \{ C_1,C_2,\dots,C_n \}$, and \emph{irreducible} if otherwise \cite[Chapter 13, Section 2]{FM}.
    Such $\{ C_1,C_2,\dots,C_n \}$ is called an invariant reduction system.
    
    By the classification in \cite{Ishizaka}, every hyperelliptic periodic automorphism of $\Sigma_{g}$ is the power of an irreducible periodic automorphism.
    In turn, as we will see below, this is not the case for periodic automorphisms which commute with an involution $\iota_{g}$ such that $\Sigma_{g}/\langle \iota_{g} \rangle$ is homeomorphic to $T^2$ (see Examples \ref{ex:redodd}, \ref{ex:redeven} and Proposition \ref{prop:redpf}).
    In this article, we focus on the classification in the irreducible case.
    Irreducible periodic automorphisms of $\Sigma_g$ have been studied from various viewpoints (see, e.g., \cite{Hirose2,Hirose-Kasahara,Ge-Ra,Dhanwani}).

    By Kerckhoff's theorem \cite[Theorem 5]{Kerckhoff}, it is known that any finite subgroup of $\operatorname{Mod}(\Sigma_g)$ can be lifted to the group of orientation-preserving diffeomorphisms of $\Sigma_g$.
    Therefore, when considering periodic automorphism of $\Sigma_g$ from now on, we will fix one representative diffeomorphism for discussion.
    In particular, note that thereafter, we will only consider the case where periodic diffeomorphisms are orientation-preserving.

    To state the result, let us present an example of an irreducible periodic diffeomorphism $h_{n, p}$ of an oriented closed surface parametrized by a positive integer $n \geq 3$ and $p \in \{ 1,\dots, n-1 \}$, which was considered also in \cite[Section 2]{PRS19}, \cite[Chapter 2]{Dh20}, \cite{Takahashi} and \cite{TN}.
    Consider a regular $2n$-gon.
    Let $\alpha_i$ and $\beta_i\ ( i = 0, \dots, n-1 )$ be the edges of the $2n$-gon as shown in Figure \ref{fig:Sig}.
    For each $i = 0,1, \dots, n-1$, equip $\alpha_i$ the clockwise orientation and $\beta_i$ the counter-clockwise orientation, respectively.
    By identifying $\alpha_i$ with $\beta_j$ so that the orientaions are compatible for every pair of $i$ and $j \in \{ 0, \dots, n-1 \}$ with $j - i \equiv p \mod n$, we obtain a surface $\Sigma$ of genus $\displaystyle g = \frac{n-\opn{gcd}(n, p) - \opn{gcd}(n, p+1) + 1}{2}$ (see Proposition \ref{prop:val}).
    We denote $h_{n,p}$ the diffeomorphism of $\Sigma$ induced from the clockwise $\frac{2\pi}{n}$-rotation of the $2n$-gon.
    To simplify the figures, both $\alpha_{0}$ and $\beta_{p}$ will be denoted by $\alpha$, and we omit the symbols of other edges in the rest of the paper.

		\begin{figure}[ht]
			\centering
			\includegraphics[width=0.5\textwidth]{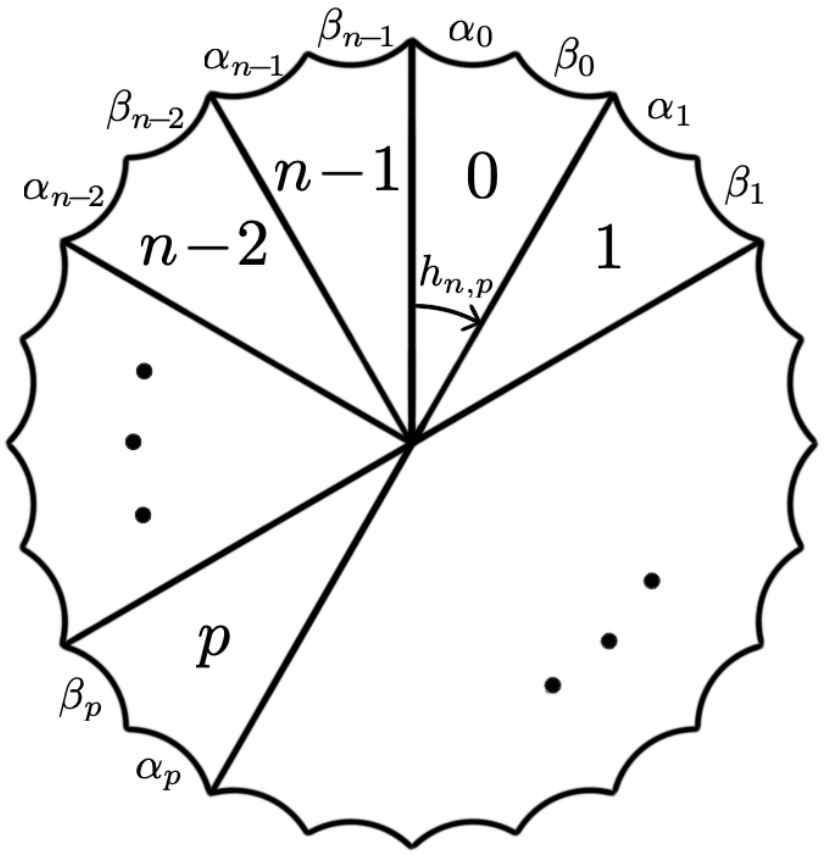}
			\caption{A diffeomorphism $h_{n,p}$}
			\label{fig:Sig}
		\end{figure}

        \begin{ex}[{\cite[Section 1]{Takahashi}}]
        \label{ex:redodd}
            We can construct an example of a reducible periodic diffeomorphism on $\Sigma_{2g+1}$ which commutes with an involution $\iota_{2g+1}$ from $h_{4g, 2g}$ and $h_{4g, 2g}^{4g-1}$ by a connected sum (Figure \ref{fig:connsum}):
            Remove small disk neighborhoods of the fixed points of $h_{4g, 2g}$ and $h_{4g, 2g}^{4g-1}$.
            Let us denote the circle boundaries $\partial_{1,1}, \partial_{1,2}, \partial_{2,1}$ and $\partial_{2,2}$ as shown in Figures \ref{fig:redoddbd1} and \ref{fig:redoddbd2}.
            For $i=1$ and $2$, we glue $\partial_{1,i}$ to $\partial_{2,i}$ (see Figure \ref{fig:connsum}).
            Then, we get a periodic diffeomorphism $F_1$ on $\Sigma_{2g+1}$ of period $4g$.
            The restriction of $F_1^{2g}$ to each of two surfaces coincides with $h_{4g,2g}^{2g}$ and $h_{4g,2g}^{2g(4g-1)}$, respectively.
            By Theorem 1 in \cite{TN}, the diffeomorphism $h_{4g,2g}$ is hyperelliptic periodic diffeomorphism and we have $\langle h_{4g,2g}, I \rangle = \langle h_{4g,2g} \rangle$ (This group is exactly the group $G_2$ of Theorem 1 in \cite{TN}).
            Since both $h_{4g,2g}^{2g}$ and $h_{4g,2g}^{2g(4g-1)}$ are hyperelliptic involutions, we can see that $F_1^{2g}$ is an involution whose quotient space is a torus.
            Since $F_1^{2g}\circ F_1 = F_1^{2g+1} = F_1 \circ F_1^{2g}$, the map $F_1^{2g}$ commutes with $F_1$.
        \end{ex}
    
        \begin{figure}[ht]
		\begin{tabular}{cc}
			\begin{minipage}{0.45\textwidth}
			    \captionsetup{width=.8\linewidth}
				\centering
				\includegraphics[width=0.9\textwidth]{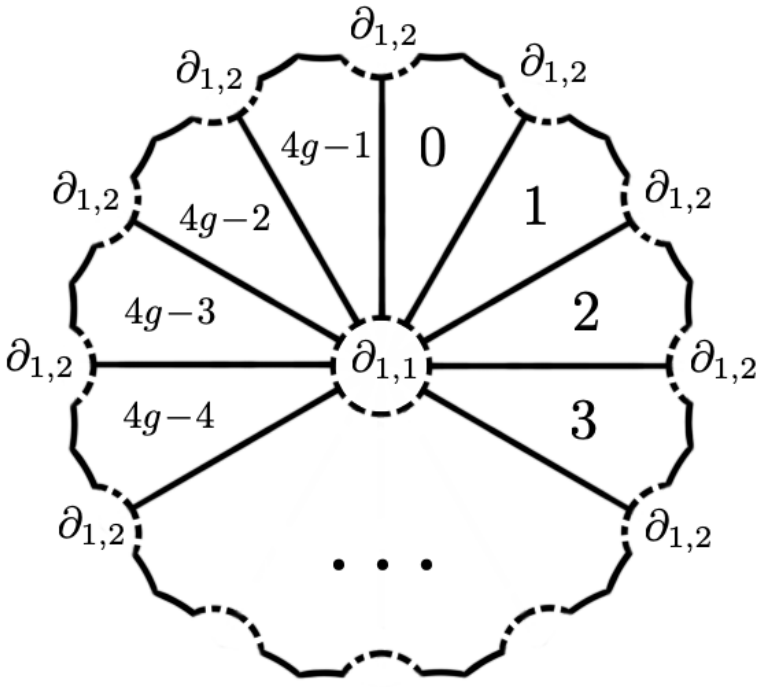}
				\caption{Boundaries $\partial_{1,1}$ and $\partial_{1,2}$}
				\label{fig:redoddbd1}
			\end{minipage}\hfill
			\begin{minipage}{0.45\textwidth}
				\centering
				\captionsetup{width=.8\linewidth}
				\includegraphics[width=0.9\textwidth]{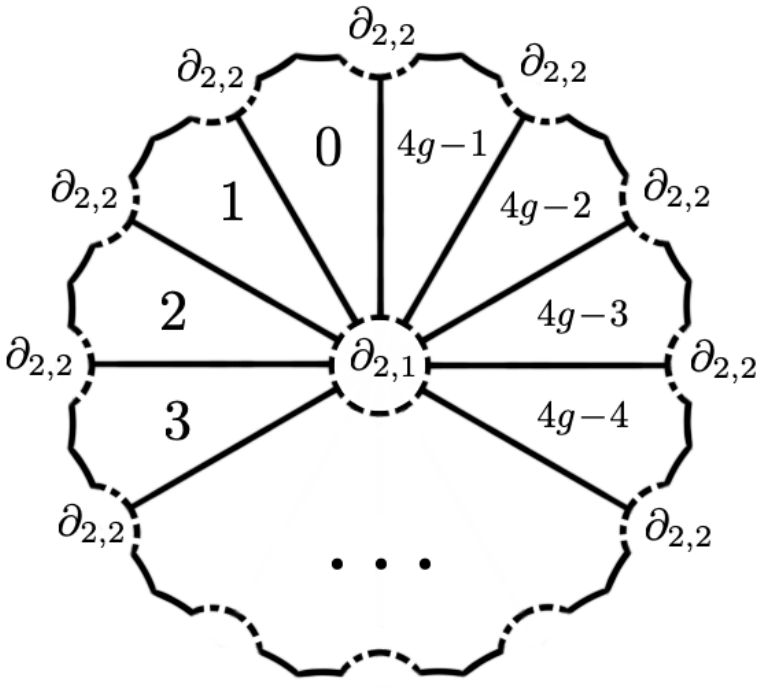}
				\caption{Boundaries $\partial_{2,1}$ and $\partial_{2,2}$}
				\label{fig:redoddbd2}
			\end{minipage}
		\end{tabular}
		\end{figure}
		
		\begin{figure}[ht]
			\centering
			\captionsetup{width=.8\linewidth}
			\includegraphics[width=0.5\textwidth]{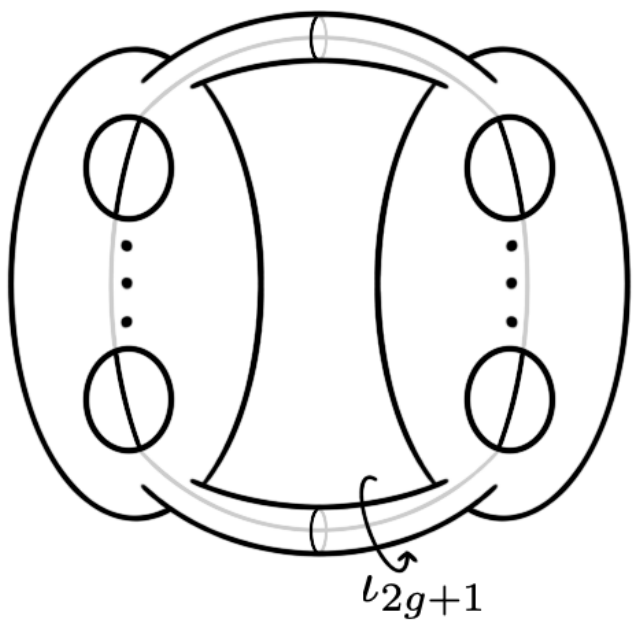}
			\caption{Glue the boundaries of the complements of disk neighborhoods of the fixed points of $(\Sigma_g, h_{4g,2g})$ and $(\Sigma_g, h_{4g,2g}^{4g-1})$ to obtain a diffeomorphism that commutes with an involution $\iota_{2g+1}$}
			\label{fig:connsum}
		\end{figure}
	
    The genus of the surface is odd in the previous example.
    In the next, we present an example with an even genus.
        
        \begin{ex}
        \label{ex:redeven}
            Remove small disk neighborhoods of the fixed points of $h_{2g+1,1}$ and $h_{2g+1,1}^{2g}$.
            We call the circle boundaries $\partial_{1,1}, \partial_{1,2},\partial_{1,3}, \partial_{2,1}, \partial_{2,2}$ and $\partial_{2,3}$ as shown in Figures \ref{fig:redevenbd1} and \ref{fig:redevenbd2}.
            Then, glue the circle boundary $\partial_{1,i}$ to $\partial_{2,i}$ for $i = 1,2,$ and $3$ (see Figure \ref{fig:redevenglue}).
            Then we get a periodic diffeomorphism $F_2$ on $\Sigma_{2g+2}$ of period $2g+1$ (see Lemma \ref{lem:h2g+1}).
            Since the diffeomorphisms $h_{2g+1,1}$ and $h_{2g+1,1}^{2g}$ are hyperelliptic (see Appendix \ref{app:A}) and the intersection of $\operatorname{Fix}(h_{2g+1,1})$ and $\operatorname{Fix}(I)$ consists of a single point (see, Proposition \ref{prop:onepoint}), we can see that hyperelliptic involutions on two surfaces induce an involution on the connected sum whose quotient space is $T^2$.
            The periodic diffeomorphism $F_2$ commutes with the involution as shown in Figure \ref{fig:redevenglue}.
        \end{ex}
        
        \begin{figure}[ht]
		\begin{tabular}{cc}
			\begin{minipage}{0.45\textwidth}
			    \captionsetup{width=0.75\linewidth}
				\centering
				\includegraphics[width=0.9\textwidth]{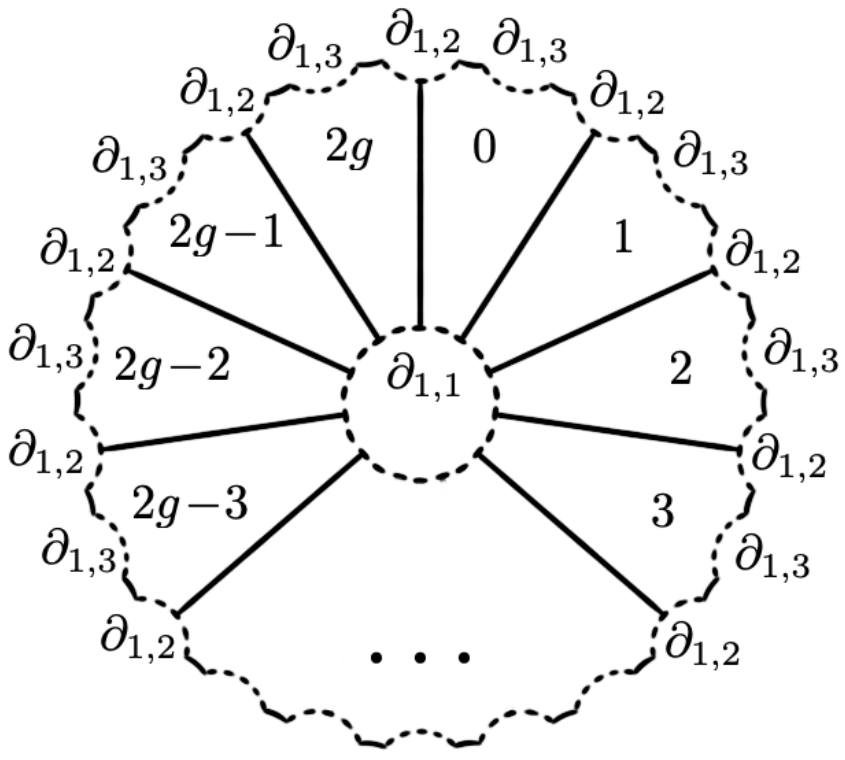}
				\caption{Boundaries $\partial_{1,1}$, $\partial_{1,2}$ and $\partial_{1,3}$}
				\label{fig:redevenbd1}
			\end{minipage}\hfill
			\begin{minipage}{0.45\textwidth}
				\centering
				\captionsetup{width=0.875\linewidth}
				\includegraphics[width=0.9\textwidth]{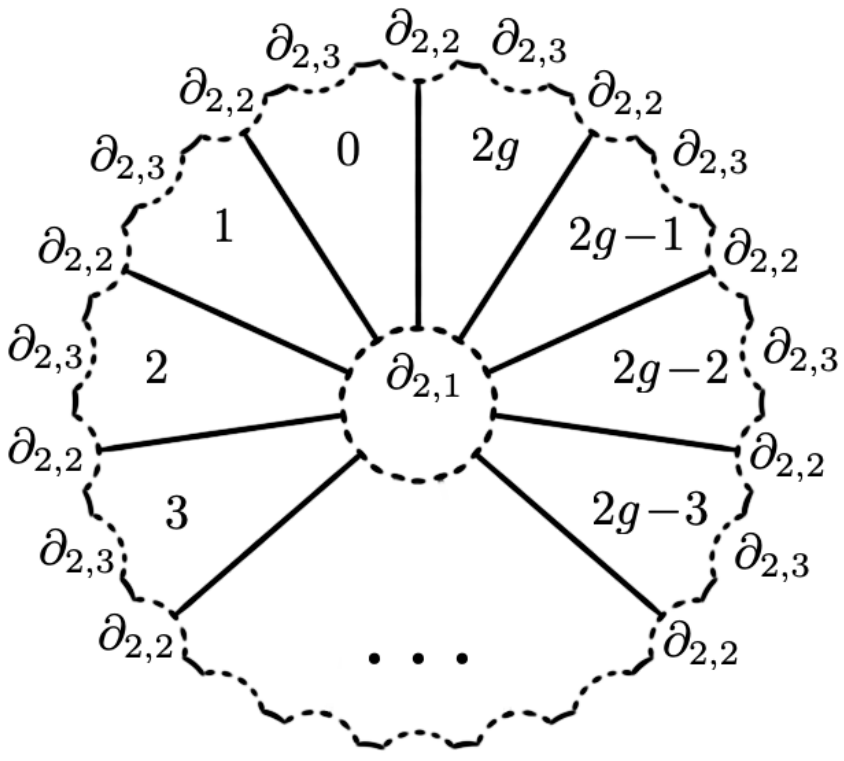}
				\caption{Boundaries $\partial_{2,1}$, $\partial_{2,2}$ and $\partial_{2,3}$}
				\label{fig:redevenbd2}
			\end{minipage}
		\end{tabular}
		\end{figure}
		
		\begin{figure}[ht]
			\centering
			\includegraphics[width=0.9\textwidth]{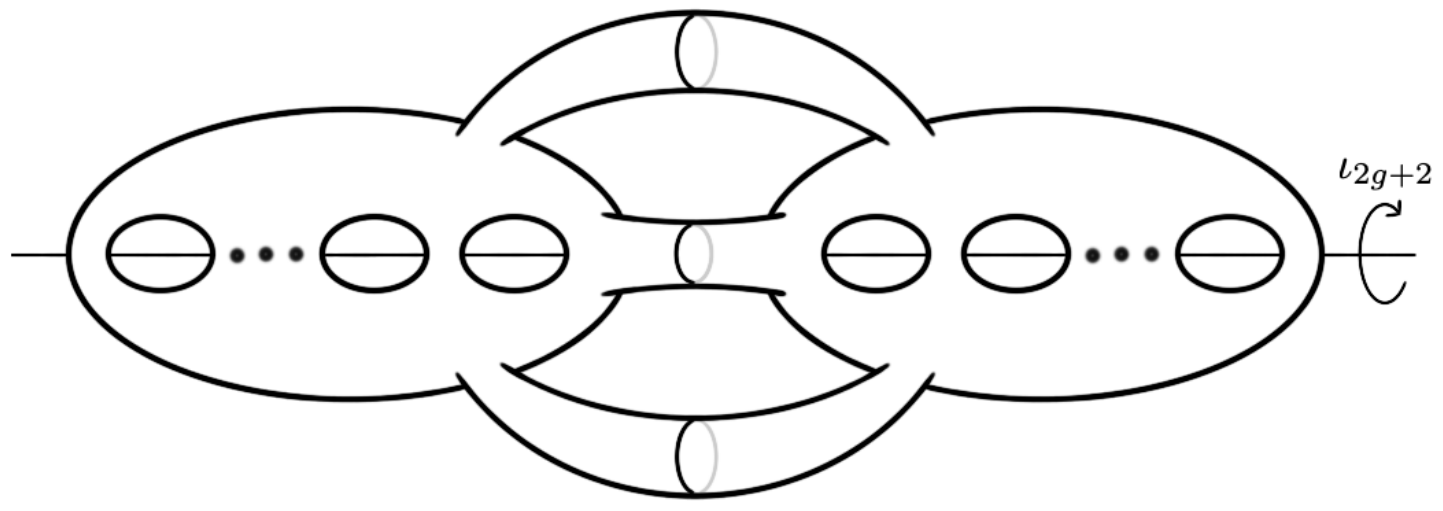}
			\caption{As in Example \ref{ex:redeven}, glue the boundaries of the complements of disk neighborhoods of the fixed points of $(\Sigma_g,h_{2g+1})$ and $(\Sigma_g, h_{2g+1}^{2g})$}
			\label{fig:redevenglue}
		\end{figure}

    Now, let $\operatorname{Diff}_{+}(\Sigma_g)$ denote the orientation-preserving diffeomorphism group of $\Sigma_g$.
    Let us recall the notion of reducibility and irreducibility of a subgroup of the diffeomorphism group of surfaces \cite{Gilman} in the case where the surface is closed.
    Recall that a partition of a closed surface is a set of mutually disjoint simple closed curves such that no two are freely homotopic and none of which is homotopic to a point.
    A subgroup $G$ in $\operatorname{Diff}_+(\Sigma_g)$ is defined to be reducible if it has an invariant partition and is irreducible otherwise.
    
    Let us explain the relation between this reducibility and the reducibility of periodic automorphism used above.
    The reducibility of a subgroup of $\operatorname{Mod}(\Sigma_g)$ is defined by the existence of an invariant reduction system on $\Sigma_g$ as in the case of periodic automorphisms.(see, \cite{Ivanov01})
    This definition of the reducibility of subgroups of $\operatorname{Mod}(\Sigma_g)$ naturally generalizes that of periodic automorphisms.
    Indeed, by definition, a periodic automorphism $\varphi$ is reducible if and only if the finite group generated by $\varphi$ is reducible.
    The definition of reducibility of finite group actions on $\Sigma_g$ is compatible with that of finite subgroups of $\operatorname{Mod}(\Sigma_g)$.
    See the last paragraph of Section \ref{sec:2} for a short argument that shows the equivalence of reducibility of a finite subgroup $G$ of $\operatorname{Diff}_+(\Sigma_g)$ and the reducibility of the image of $G$ under the projection $\operatorname{Diff}_+(\Sigma_g) \to \operatorname{Mod}(\Sigma_g)$.
    In the following, we will consider the irreducible finite subgroup of $\operatorname{Diff}_+(\Sigma_g)$.
   
    Let us state the main result of this article.
    In this result, we classify irreducible periodic diffeomorphisms which commute with $\iota_g$. 
	
		\begin{thm}
		\label{thm:irr1}
			Let $I$ and $\iota_g$ be a hyperelliptic involution and an involution of $\Sigma_g$ such that $\Sigma_g/ \langle \iota_g \rangle$ is homeomorphic to a torus (Figure \ref{fig:iotag}), respectively.
			For any periodic diffeomorphism $f$ of $\Sigma_{g}$ which commutes with an involution $\iota_g$, if the subgroup $G=\langle f, \iota_g \rangle$ of $\opn{Diff}_{+}(\Sigma_{g})$ is irreducible, then $G$ is conjugate in $\operatorname{Diff}_{+}(\Sigma_g)$ to a subgroup of one of the followings:
                \begin{enumerate}
                    \item $\langle h_{6,1}, I \rangle$ ($g=2$ and see Figures \ref{fig:Ionh61} and \ref{fig:thm3}),
                    \item $\langle h_{8,1} \rangle$ ($g=3$ and see Figure \ref{fig:thm3}),
                    \item $\langle h_{8,5} \rangle$ ($g=3$ and see Figure \ref{fig:thm3}),
                    \item $\langle h_{12,2} \rangle$ ($g=4$ and see Figure \ref{fig:thm3}),
                    \item $\langle h_{12,3} \rangle$ ($g=3$ and see Figure \ref{fig:thm3}).
                \end{enumerate}
		\end{thm}
		
	It is remarkable that there are only finite cases in the classification of Theorem \ref{thm:irr1}.
	It contrasts with the classification of hyperelliptic periodic classes, where we have three infinite families of conjugacy classes.
	    
	    \begin{figure}[ht]
		\begin{tabular}{cc}
			\begin{minipage}{0.45\textwidth}
				\centering
				\includegraphics[width=1\textwidth]{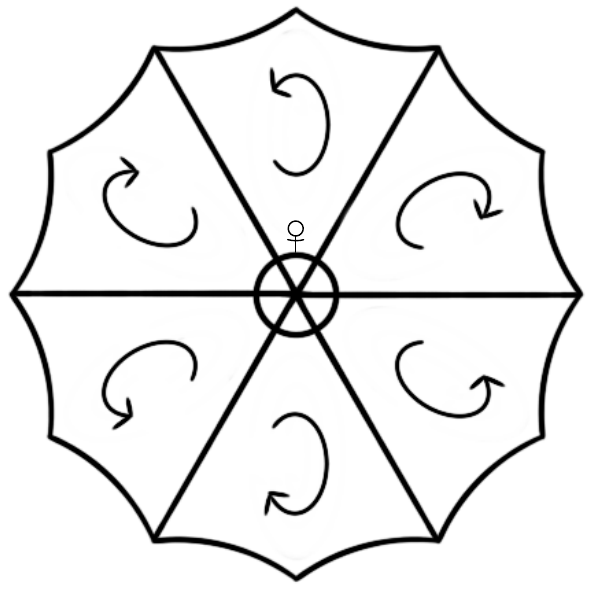}
			\end{minipage}\quad$\xrightarrow{\ I\ }$\quad
			\begin{minipage}{0.45\textwidth}
				\centering
				\includegraphics[width=1\textwidth]{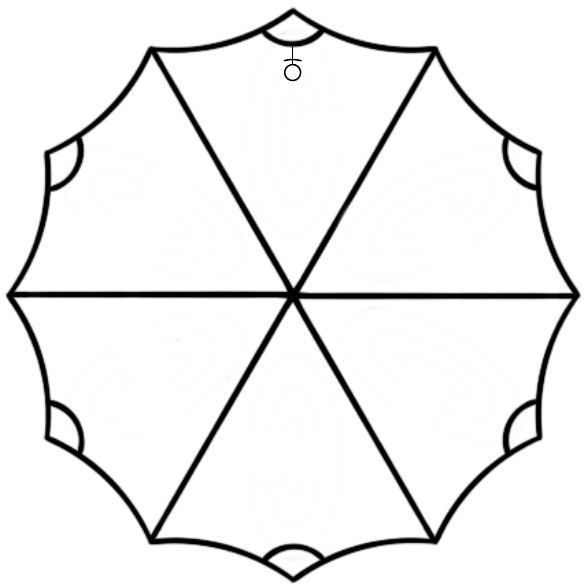}
			\end{minipage}
		\end{tabular}
		\caption{The action of $I$ on $(\Sigma_2, h_{6,1})$}
		\label{fig:Ionh61}
		\end{figure}
        
		\begin{figure}[ht]
		\begin{tabular}{ccc}
			\begin{minipage}{0.33\textwidth}
				\centering
				\includegraphics[width=1\textwidth]{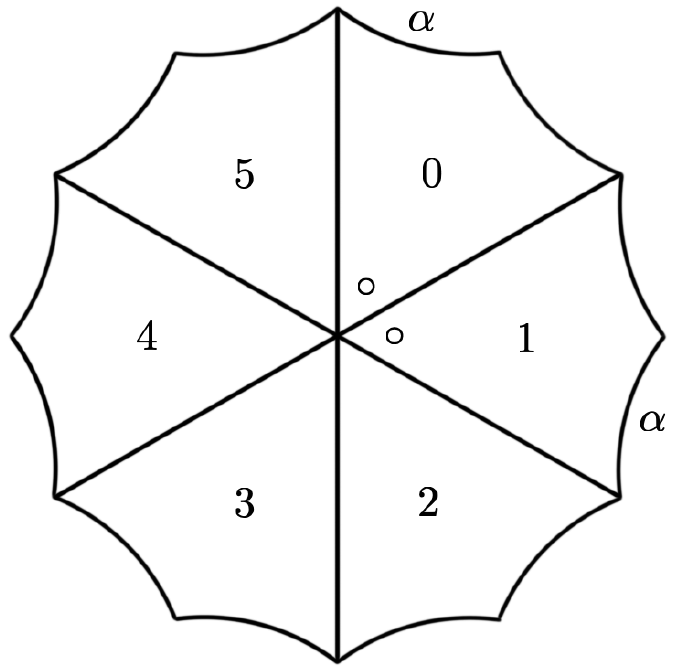}
				\caption*{$h_{6,1}$}
			\end{minipage}\hfill
			\begin{minipage}{0.33\textwidth}
				\centering
				\includegraphics[width=1\textwidth]{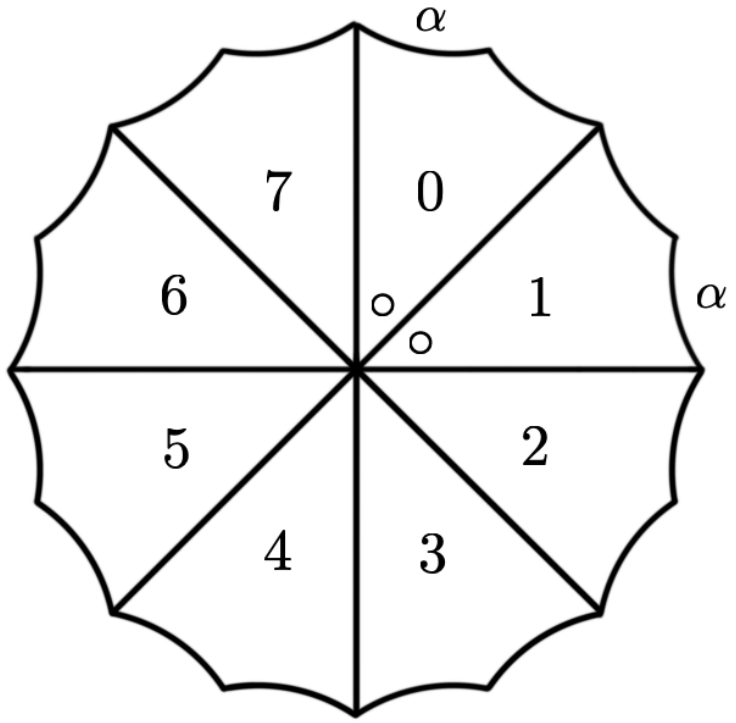}
				\caption*{$h_{8,1}$}
			\end{minipage}\hfill
			\begin{minipage}{0.33\textwidth}
				\centering
				\includegraphics[width=1\textwidth]{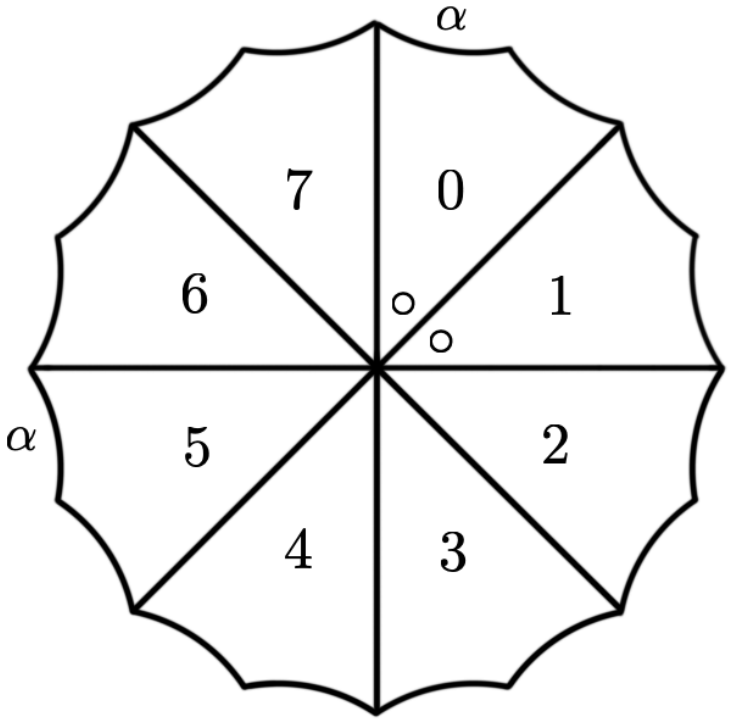}
				\caption*{$h_{8,5}$}
			\end{minipage} 
		\end{tabular}
		\begin{tabular}{cc}
			\begin{minipage}{0.33\textwidth}
				\centering
				\includegraphics[width=1\textwidth]{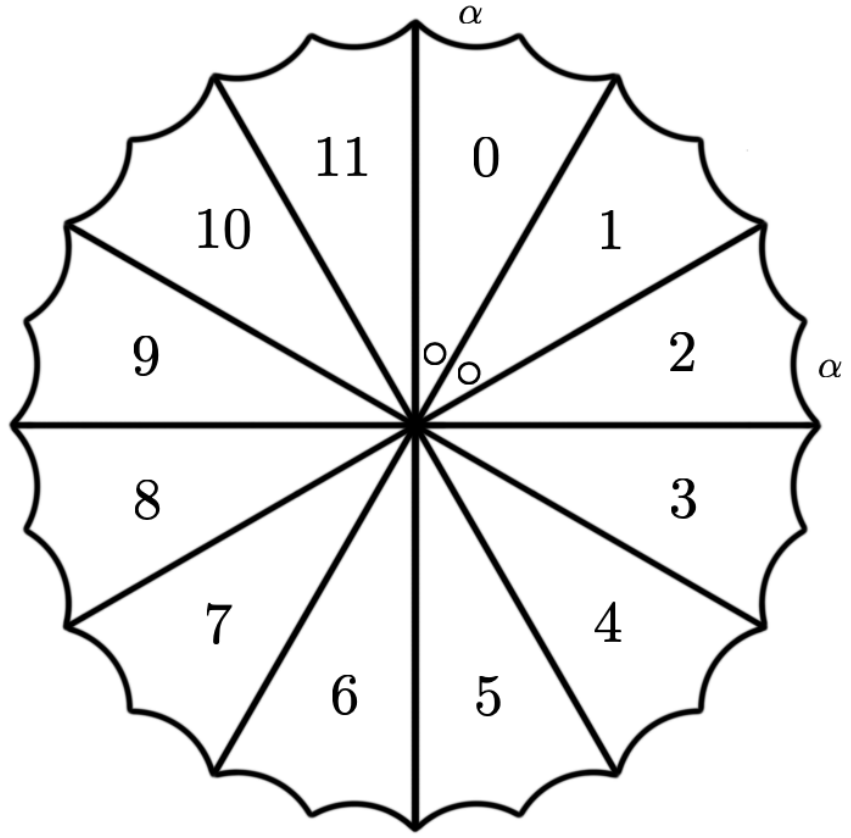}
				\caption*{$h_{12,2}$}
			\end{minipage} \hfill
			\begin{minipage}{0.33\textwidth}
				\centering
				\includegraphics[width=1\textwidth]{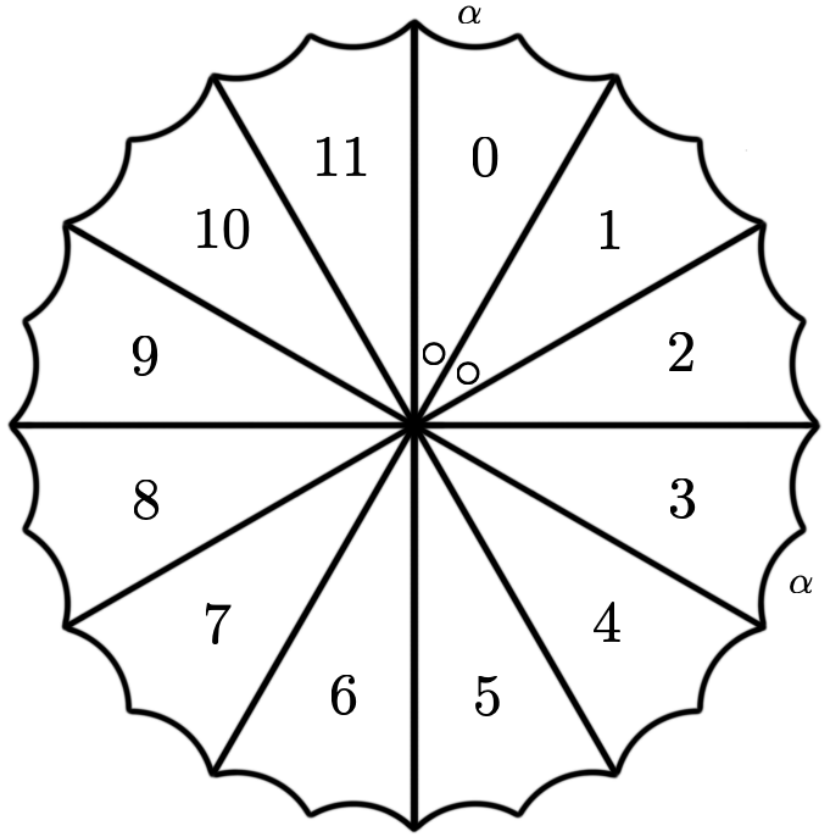}
				\caption*{$h_{12,3}$}
			\end{minipage}
		\end{tabular}
            \caption{The diffeomorphisms appearing in Theorem \ref{thm:irr1}\\ Each barycenter is a fixed point, and each small circle represents how the angles move}
            \label{fig:thm3}
		\end{figure}


\section{Periodic diffeomorphisms and their total valencies}
\label{sec:2}
    
    We first recall the classification of the conjugacy classes of periodic diffeomorphisms on $\Sigma_{g}$ in terms of the total valency introduced in \cite{Ashikaga-Ishizaka}.
    Let $f \in \opn{Diff}_+(\Sigma_{g})$ be a periodic diffeomorphism of order $n$ and $C$ the cyclic subgroup of $\opn{Diff}_{+}(\Sigma_{g})$ generated by $f$.
    A $C$-orbit $Cx = \{ f^i(x) \mid 0 \leq i \leq n-1 \}$ is called \emph{multiple} if $|Cx| < n$ and is called \emph{free} if $|Cx| = n$.
    The points in multiple orbits are \emph{multiple points} of $f$.
    Let $M_f$ be the set of all multiple points of $f$ and $p : \Sigma_g \to \Sigma_g / C$ the canonical projection.
    We call $p(M_f)$ the \emph{branch locus} of $p$ and its elements the \emph{branched points} of $p$.
    We say that $p$ is \emph{branched} at the points in $p(M_f)$.
    For a subgroup $G$ of $\operatorname{Diff}_+(\Sigma_g)$ and $x \in \Sigma_g$, $\operatorname{Stab}_{G}(x) = \{ f \in G \mid f(x) = x \}$ is called the \emph{stabilizer subgroup} of $G$ with respect to $x$.
    For $x \in \Sigma_g$, the minimum positive integer $\ell$ such that $f^{\ell}(x)=x$ is called the \emph{period} of $Cx$.
    Then, we have $\operatorname{Stab}_{C} (x) = \langle f^\ell \rangle $.
    For a multiple point $x \in M_f$ with period $\ell$, the integer $\lambda = \displaystyle \frac{n}{\ell}$ is called the \emph{branching index} of the branched point $p(x)$.
    Then, there uniquely exists $\mu \in \{1,2, \dots, \lambda - 1\}$ such that the restriction of $f^\ell$ to a small disk neighborhood of $x$ is the clockwise $\frac{2\pi\mu}{\lambda}$-rotation.
    Since the group $\langle f^\ell \rangle$ of order $\lambda$ acts effectively on the disk neighborhood of $x$ and a generator $f^\ell$ is $\frac{2\pi\mu}{\lambda}$-rotation, $\mu$ and $\lambda$ are coprime.
    Since $\mu$ and $\lambda$ are coprime, there uniquely exists an integer $\theta \in \{1,2, \dots, \lambda - 1\}$ such that $\mu \theta \equiv 1 \mod \lambda$.
    The \emph{valency} of $Cx$ is defined by $\frac{\theta}{\lambda}$.
    Let $\frac{\theta_1}{\lambda_1}, \frac{\theta_2}{\lambda_2}, \dots, \frac{\theta_s}{\lambda_s}$ be the valencies of all multiple orbits of $f$.
    The data $\left[ \, g,n \,; \frac{\theta_1}{\lambda_1} + \frac{\theta_2}{\lambda_2} + \dots + \frac{\theta_s}{\lambda_s} \, \right]$ is called the \emph{total valency} of $f$.
    
    By the following theorem, total valencies determine periodic diffeomorphisms of $\Sigma_{g}$ up to conjugacy.
    
        \begin{thm}[{Nielsen \cite[Section 11]{Nielsen}, see also \cite[Section 1.3]{Ashikaga-Ishizaka}, \cite[Theorem 2.1]{Hirose}}]
        \label{thm:Nielsen}
            Let $f, f'$ be periodic diffeomorphisms of $\Sigma_{g}$ with total valencies $\left[\,g,n \,;\, \frac{\theta_1}{\lambda_1} + \frac{\theta_2}{\lambda_2} + \dots +\frac{\theta_s}{\lambda_s} \,\right]$ and $\left[\,g,n' \,;\,\frac{\theta_1'}{\lambda_1'} + \frac{\theta_2'}{\lambda_2'} + \dots +\frac{\theta_{s'}'}{\lambda_{s'}'} \,\right]$, respectively.
            $f$ is conjugate to $f'$ if and only if the following are satisfied:
                \begin{enumerate}
                \renewcommand{\labelenumi}{(\roman{enumi})}
                    \item $s=s'$,
                    \item $n=n'$,
                    \item after changing indices, we have $\frac{\theta_i}{\lambda_i}=\frac{\theta_i'}{\lambda_i'}$ for $i =1,2,\dots,s$.
                \end{enumerate}
        \end{thm}
    
    By Theorem \ref{thm:Nielsen}, we can identify the total valency of a periodic diffeomorphism with its conjugacy class.
    We will use the following well-known observations of Nielsen.
    
        \begin{prop}[{Nielsen \cite[Equation (4.6)]{Nielsen}}]
        \label{prop:Nielsenint}
            For any periodic diffeomorphism, the sum of the valencies of all multiple orbits is an integer.
        \end{prop}
        
        \begin{rem}
        \label{rem:inv}
            Every multiple orbit of an involution is a fixed point whose valency is $\frac{1}{2}$.
            Thus, by Proposition \ref{prop:Nielsenint}, the number of fixed points of an involution is even.
        \end{rem}
    
    In general, the composite of two periodic diffeomorphisms may not be periodic, and hence, the product of total valencies does not make sense.
    But one can define powers of the total valency of a periodic diffeomorphism $f$ by the total valency of powers of $f$.
    
    Let us recall the well-known Riemann-Hurwitz formula for the convenience of the reader.
        
        \begin{prop}[{Riemann-Hurwitz formula}]
        \label{prop:RH}
            Consider an $n$-fold branched covering from $\Sigma_g$ to $\Sigma_{g'}$ with branching indices $\lambda_1, \dots, \lambda_s$.
            Then we have
                \[
                    2g-2 = n \left( 2g' -2 + \displaystyle\sum_{1 \leq i \leq s} \left( 1 - \frac{1}{\lambda_i} \right) \right).
                \]
        \end{prop}

    We will use the following result due to Harvey, which gives a necessary condition for the existence of finite covering maps that have given branching indices.
	
        \begin{prop}[{Harvey \cite[Theorem 4]{Harvey}}]
        \label{prop:Harvey}
            Assume $g>1$ and let $n \in \ZZ_{\geq 2}$ and $\lambda_1, \dots, \lambda_s \in \ZZ_{\geq 2}$.
            Set $M = \opn{lcm}(\lambda_1,\lambda_2,\dots , \lambda_s)$.
            If there is an $n$-fold cyclic covering from $\Sigma_g$ to $\Sigma_{g'}$ with branching indices $\lambda_1,\dots, \lambda_s$, then the following conditions are satisfied:
                \begin{enumerate}
                \renewcommand{\labelenumi}{(\roman{enumi})}
                    \item $\opn{lcm}(\lambda_1,\lambda_2,\dots, \hat{\lambda_i}, \dots , \lambda_s) = M$ for all $i \in \{ 1,2,\dots, s\}$, where $\hat{\lambda_i}$ denotes the omission of $\lambda_i$.
                    \item $M$ divides $n$, and if $g' = 0$, then $M=n$.
                    \item $s \neq 1$, and, if $g'=0$, then $s \geq 3$.
                \end{enumerate}
        \end{prop}
    
    We will use the following characterizations of reducibility of periodic diffeomorphisms.
    
    \begin{thm}[{Kasahara \cite[Theorem 4.1]{Kasahara}}]
    \label{thm:Kasahara}
        Let $f \in \operatorname{Diff}_+(\Sigma_g)$ be a periodic diffeomorphism of order $n$.
        If $f$ is irreducible, then $n \geq 2g+1$.
    \end{thm}
    
    \begin{thm}[{Wiman, see \cite[Theorem 5.1]{Hirose-Kasahara}}]
    \label{thm:Wiman}
        Let $f \in \operatorname{Diff}_+(\Sigma_g)$ be a periodic diffeomorphism of order $n$.
        Then, $n \leq 4g+2$.
    \end{thm}
    
    Moreover, for finite group actions, we will also use the following characterization due to Gilman.
        
        \begin{thm}[{Gilman \cite[Lemma 3.9]{Gilman}}]
        \label{thm:Gilman}
            Let $G$ be a finite subgroup of $\opn{Diff}_+(\Sigma_g)$ and $M_G$ be the union of all the multiple orbits of $G$.
            Then $G$ is irreducible if and only if $(\Sigma_g \backslash M_G) / G$ is homeomorphic to $S^2 \backslash \{3 \text{-points}\}$.
        \end{thm}
        
        \begin{rem}
            Let $G$ be a finite subgroup $\operatorname{Diff}_+(\Sigma_g)$.
            Let us prove that $\rho$ is reducible if and only if the image of $G$ under the projection $\operatorname{Diff}_+(\Sigma_g) \to \operatorname{Mod}(\Sigma_g)$.
            We can show a stronger result: By Kerckhoff's theorem \cite{Kerckhoff}, every finite subgroup $H$ of $\operatorname{Mod}(\Sigma_g)$ can be lifted to $\operatorname{Diff}_+(\Sigma_g)$, namely, there exists a subgroup $\widetilde{H}$ of $\operatorname{Diff}_+(\Sigma_g)$ such that the restriction of the projection $\operatorname{Diff}_+(\Sigma_g) \to \operatorname{Mod}(\Sigma_g)$ to $\widetilde{H}$ is an isomorphism from $\widetilde{H}$ to $H$. Moreover, $\widetilde{H}$ preserves a hyperbolic metric on $\Sigma_g$. We will prove that $H$ is reducible if and only if $\widetilde{H}$ is reducible.
        
            If $\widetilde{H}$ is reducible, then there is an $\widetilde{H}$-invariant partition $\mathcal{T}$ on $\Sigma_g$.
            The isotopy classes of $\mathcal{T}$ give a reduction system on $\Sigma_g$ for $H$, which implies the reducibility of $H$.
            Conversely, assume that $H$ is reducible.
            Then there exists an $H$-invariant reduction system $\mathcal{C} =\{C_1, \dots, C_n\}$.
            Since $\Sigma_g$ is hyperbolic, $C_i$ is represented by a unique simple closed geodesic $D_i$ on $\Sigma$ for $i=1, \dots, m$ (see, for example, \cite[Lemmas 3.3 and 3.4]{Poenaru}).
            By a classical result (for example, see \cite[Th\'eor\`eme 15]{Poenaru}), a pair of simple closed geodesics on a closed hyperbolic surface intersects with each other at the minimal number of points in their isotopy classes.
            Since the geometric intersection number of $C_i$ and $C_j$ is $0$ for $i \neq j$, the geodesics $D_1, \dots, D_m$ are mutually disjoint.
            Since $\widetilde{H}$ preserves the hyperbolic metric on $\Sigma_g$ and $D_i$ is the unique geodesic that represents $C_i$ up to free homotopy, it follows that $\{D_1, \dots, D_m\}$ is $\widetilde{H}$-invariant.
            Therefore, having an $\widetilde{H}$-invariant partition $\{D_1, \dots, D_m\}$, the action $\widetilde{H}$ is reducible.
        \end{rem}

\section{Preliminaries}
    
	Let us compute the total valency of $h_{n,p}$ defined in the introduction.

		\begin{prop}[{\cite[Section 2]{PRS19}, \cite[Theorem 2.2.4]{Dh20} and \cite[Proposition 4.1]{Takahashi}}]
		\label{prop:val}
			The total valency of $h_{n,p}$ is $\left[\, g, n\, ; \, \frac{1}{n} + \frac{p}{n} + \frac{n-p-1}{n} \,\right]$, where
                \[
                    g=\displaystyle\frac{n-\opn{gcd}(n, p) - \opn{gcd}(n, p+1) + 1}{2}.
                \]
		\end{prop}
		
		\begin{proof}
			Let $P$ be the $2n$-gon shown in Figure \ref{fig:Sig}.
                As shown below, $h_{n,p}$ has $3$ multiple orbits: The barycenter $x$ of $P$, the middle points $\{ y_0, y_1, \dots, y_{n-1} \}$ of the edges of $P$ and the vertices $\{ z_0, z_1, \dots, z_{n-1} \}$ of $P$ (see Figure \ref{fig:barycenter}).
			Thus, the total valency of $h_{n, p}$ is of the form $\left[ \,  g, n \,;\, \frac{\theta_1}{\lambda_1}+ \frac{\theta_2}{\lambda_2}+ \frac{\theta_3}{\lambda_3} \, \right]$ for some mutually coprime integers $\lambda_i$ and $ \theta_i$ such that $\lambda_i$ divides $n$ for $i = 1, 2, 3$.
			
			\begin{figure}[ht]
			\begin{tabular}{cc}
			     \begin{minipage}{0.45\textwidth}
			        \captionsetup{width=.8\linewidth}
			        \centering
	    		    \includegraphics[width=0.7\textwidth]{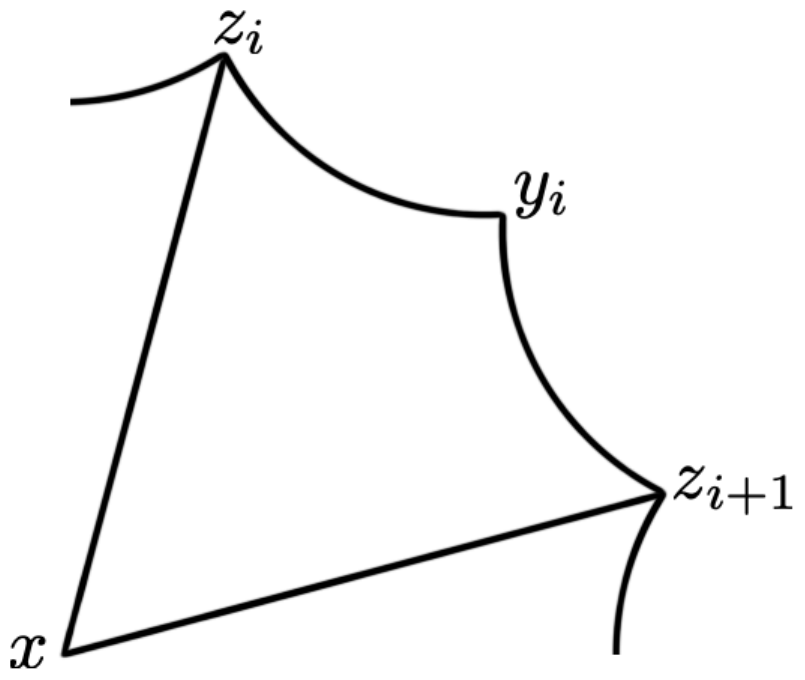}
        			\caption{The barycenter $x$, the middle points $y_i$ of edges and the vertices $z_i$ of $P$}
        			\label{fig:barycenter}
			     \end{minipage}\hfill
			     \begin{minipage}{0.45\textwidth}
			        \captionsetup{width=.75\linewidth}
			        \centering
			        \includegraphics[width=0.7\textwidth]{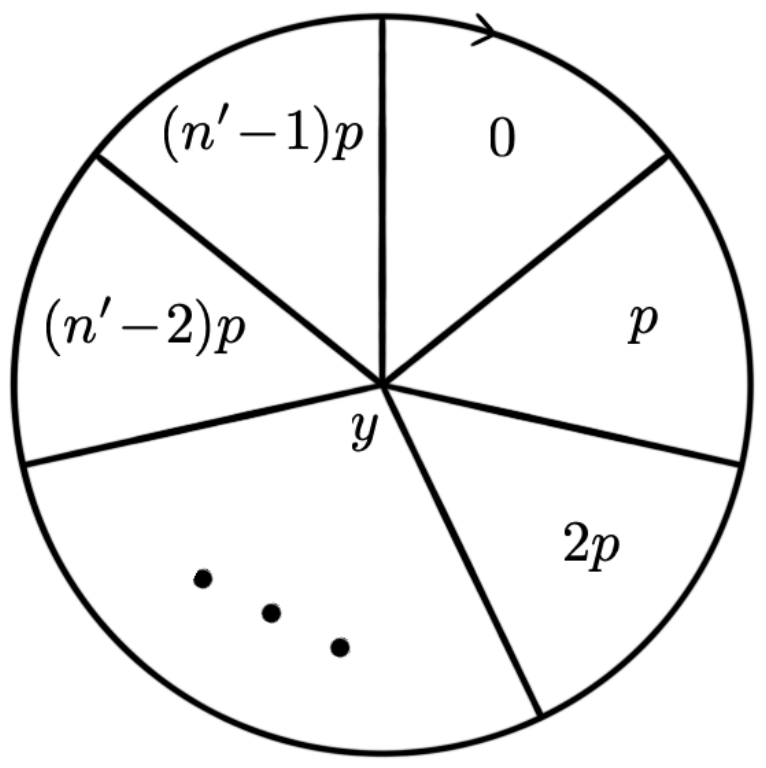}
				    \caption{The cyclic order of fundamental domains of $h_{n,p}$ around $y$}
				    \label{fig:Sigval}
			     \end{minipage}
			\end{tabular}
	    	\end{figure}
			
			Since the action is $\frac{2\pi}{n}$ rotation on the $2n$-gon, the valency of the fixed point $x$ is $\frac{1}{n}$.
			Hence, we can assume $\frac{\theta_1}{\lambda_1} = \frac{1}{n}$.
			
			Let us compute the valency $\frac{\theta_2}{\lambda_2}$ of the multiple orbit $\{ y_0, y_1, \dots, y_{n-1} \}$.
			Set $k = \gcd (n, p)$.
			Let $n'$ and $p'$ be the integer such that $n = k n'$ and $p = k p'$, respectively.
			The period of the multiple orbits is $k$, and hence, $\lambda_2 = n'$.
			Take a point $y$ in the orbit and consider a loop that goes around the boundary of a small disk centered at $y$ as shown in Figure \ref{fig:Sigval}.
			The fundamental domains appear along this loop in the order
				\[
					0, \, p, \, 2p, \dots,  (n'-2)p, \, (n'-1)p
				\]
			modulo $n$.
			Then, there uniquely exists $\nu \in \left\{ 1, 2, \dots, n' \right\}$ such that $\nu p \equiv k \, \mod n$.
			Hence, we have $\nu p' \equiv 1 \mod n'$.
			By the definition of the valency, $\theta_2 = p'$, and hence, $\frac{\theta_2}{\lambda_2} = \frac{p'}{n'} = \frac{p}{n}$.
			
			By Nielsen's result (Proposition \ref{prop:Nielsenint}), we have
				\[
					\frac{1}{n} + \frac{p}{n} + \frac{\theta_3}{\lambda_3} \in \bb{Z}.
				\]
			Thus, we have $\frac{\theta_3}{\lambda_3} = \frac{n-p-1}{n}$.
			
			Lastly, we compute the genus $g$ of the surface.
			The branching indices of $h_{n,p}$ are $\displaystyle n, \frac{n}{\opn{gcd}(n,p)}$ and $\displaystyle\frac{n}{\opn{gcd}(n, n-p-1)}$.
			It follows from the definition of the map $h_{n,p}$ that the quotient space of $h_{n,p}$ is a sphere.
			Then, by Riemann-Hurwitz formula (Proposition \ref{prop:RH}), we have
			    \begin{align*}
			       \displaystyle 2g-2 =&\, n\left( -2 + 3 - \left( \frac{1}{n} + \frac{\opn{gcd}(n,p)}{n} + \frac{\opn{gcd}(n,n-p-1)}{n} \right) \right)\\
			            \displaystyle  =&\, n - \operatorname{gcd}(n,p) - \opn{gcd}(n,p+1) - 1,
			    \end{align*}
		    and hence, we have
			    \[
			        \displaystyle g = \frac{n - \operatorname{gcd}(n,p) - \opn{gcd}(n,p+1) + 1}{2}.
			    \]
		\end{proof}

    Applying the Riemann-Hurwitz formula (Proposition \ref{prop:RH}), we can obtain the following well-known classification of periodic diffeomorphisms on tori.

        \begin{prop}
        \label{prop:brto}
            Let $f$ be a nontrivial orientation-preserving periodic diffeomorphism on a torus $T^2$.
            Then $T^2/\langle f \rangle$ is homeomorphic to either $T^{2}$ or $S^{2}$.
            
            If $T^2/\langle f \rangle$ is homeomorphic to $T^2$, then $f$ has no multiple orbit.
            If $T^2/\langle f \rangle$ is homeomorphic to $S^2$, then $f$ is conjugate to one of the following:
                \begin{enumerate}
                    \item $h_{4,1}^2$ or $h_{6,3}^3$ whose total valency are $\displaystyle\left[\,1,2\,;\, \frac{1}{2} + \frac{1}{2} + \frac{1}{2} + \frac{1}{2}\,\right]$.
                    
                    \item $h_{4,1}$ whose total valency is $\displaystyle\left[\,1,4\,;\, \frac{1}{4} + \frac{1}{4} + \frac{1}{2}\,\right]$.
                    
                    \item $f_{4,1}^3$ whose total valency is $\displaystyle\left[\,1,4\,;\, \frac{3}{4} + \frac{3}{4} + \frac{1}{2}\,\right]$.
                    
                    \item $h_{6,3}$ whose total valency is $\displaystyle\left[\,1,6\,;\, \frac{1}{6} + \frac{1}{3} + \frac{1}{2}\,\right]$.
                    
                    \item $h_{3,1}$ or $h_{6,3}^{2}$ whose total valency are $\displaystyle\left[\,1,3\,;\, \frac{1}{3} + \frac{1}{3} + \frac{1}{3}\,\right]$.
                    
                    \item $h_{3,1}^2$ or $h_{6,3}^{4}$ whose total valency are $\displaystyle\left[\,1,3\,;\, \frac{2}{3} + \frac{2}{3} + \frac{2}{3}\,\right]$.
                    
                    \item $h_{6,3}^{5}$ whose total valency is $\displaystyle\left[\,1,6\,;\, \frac{5}{6} + \frac{2}{3} + \frac{1}{2}\,\right]$.
                \end{enumerate}
        \end{prop}
    
    The diffeomorphism $h_{3,1}$ is conjugate to the diffeomorphism $h_{6,3}^2$.
    Indeed, both are represented by a $\displaystyle\frac{\pi}{3}$-rotation of a regular hexagon (see Figures \ref{fig:h31} and \ref{fig:1232}).
    
        \begin{figure}[ht]
        \begin{tabular}{cc}
            \begin{minipage}{0.45\textwidth}
                \captionsetup{width=0.85\linewidth}
                \includegraphics[width=0.9\textwidth]{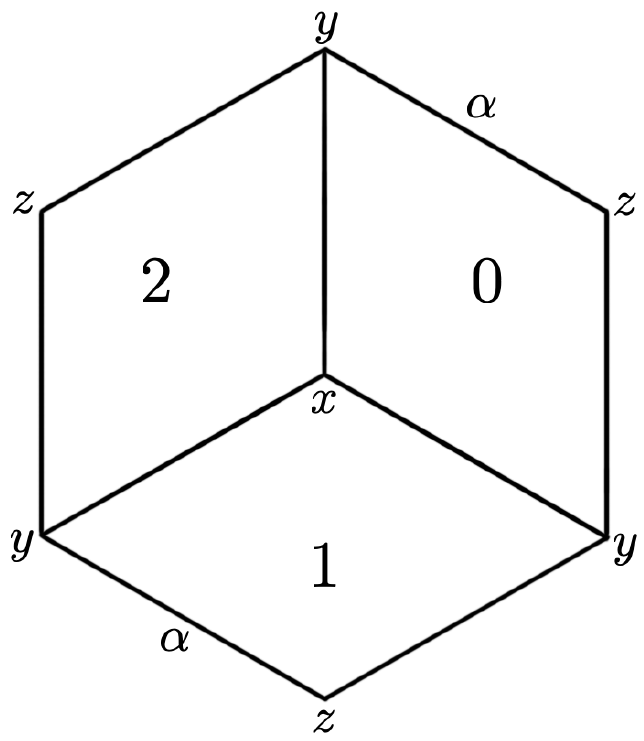}
                \caption{$h_{3,1}$}
                \label{fig:h31}
            \end{minipage}\hfill
            \begin{minipage}{0.45\textwidth}
                \captionsetup{width=0.85\linewidth}
                \includegraphics[width=1\textwidth]{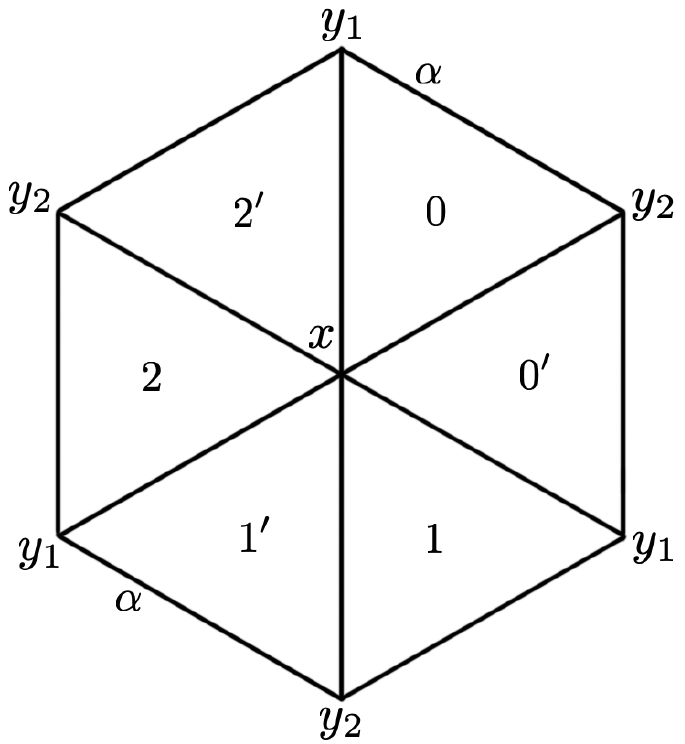}
                \caption{$h_{6,3}^2$}
                \label{fig:1232}
            \end{minipage}
        \end{tabular}
        \end{figure}

    Since any multiple-point free periodic diffeomorphism on a torus is reducible, we get the following consequence.
    
        \begin{prop}
        \label{prop:torusirr}
            Every irreducible periodic diffeomorphism on a torus is conjugate to one of the following: 
                \begin{enumerate}
                    \item $h_{3,1}$ or its inverse,
                    \item $h_{4,1}$ or its inverse, 
                    \item $h_{6,3}$ or its inverse.
                \end{enumerate}
        \end{prop}

	Here $h_{3,1}, h_{4,1}$ and $h_{6,3}$ are shown in Figures \ref{fig:h31}, \ref{fig:112} and \ref{fig:123}, respectively.

		\begin{figure}[ht]
			\centering
			\begin{tabular}{cc}
			\begin{minipage}{0.45\textwidth}
				\centering
				\includegraphics[width=1\textwidth]{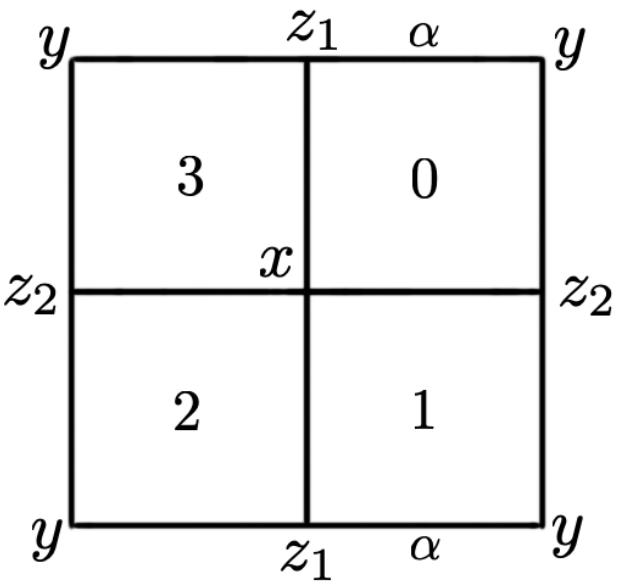}
				\caption{$h_{4,1}$}
				\label{fig:112}
			\end{minipage}\hfill
			\begin{minipage}{0.45\textwidth}
				\centering
				\includegraphics[width=0.9\textwidth]{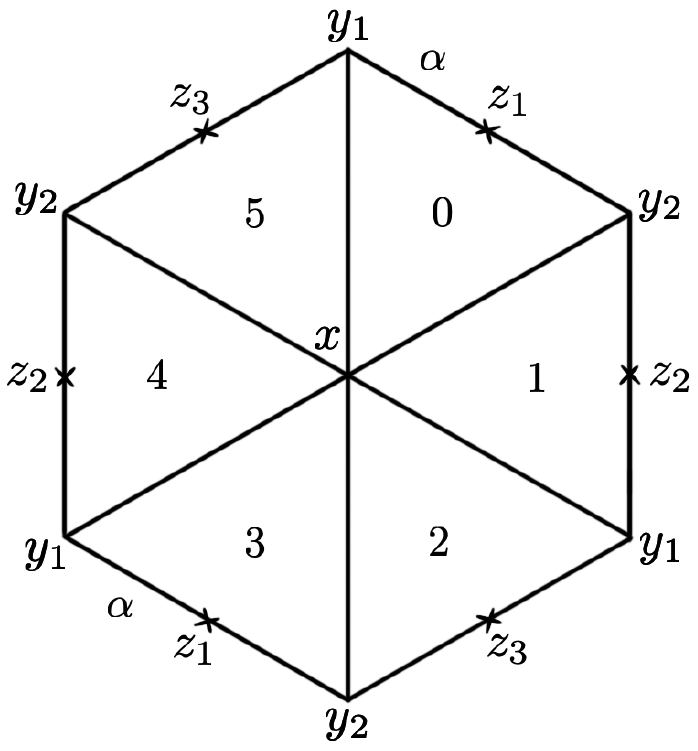}
				\caption{$h_{6,3}$}
				\label{fig:123}
			\end{minipage}
			\end{tabular}
		\end{figure}
		
\section{Cyclic branched coverings of tori}
    
    In this section, we will classify irreducible periodic diffeomorphisms that commute with an involution $\iota_{g}$ such that $\Sigma_{g}/ \langle\iota_{g}\rangle$ is a torus to prove Theorem \ref{thm:irr1}.
    
    Take an involution $\iota_g$ of $\Sigma_g$ whose quotient space is homeomorphic to a torus.
    Let $f$ be a periodic diffeomorphism of $\Sigma_g$ of order $n$, which commutes with $\iota_g$.
    Let $G$ be the subgroup of $\operatorname{Diff}_{+}(\Sigma_g)$ generated by $f$ and $\iota_g$.
    
    Let $\bar{f}$ be a periodic diffeomorphism of $\Sigma_{g} / \langle \iota_g \rangle$ induced by $f$.
    Let $n$ and $\bar{n}$ be the order of $f$ and $\bar{f}$, respectively.
    We have the following simple observation.
    
        \begin{lem}
        \label{lem:1}
            \begin{enumerate}
                \item We have $n = 2 \bar{n}$ if and only if $f^{\bar{n}} = \iota_g$.
                \item We have $n = \bar{n}$ if and only if $f^{\bar{n}} \neq \iota_g$.
            \end{enumerate}
        \end{lem}
        
        \begin{proof}
            We have $\bar{f}^{\bar{n}} = \opn{id}_{\Sigma_{g}/\langle \iota_g \rangle}$.
            Since the lifts of $\opn{id}_{\Sigma_{g} / \langle \iota_g \rangle}$ to $\Sigma_{g}$ are $\iota_g$ and $\opn{id}_{\Sigma_{g}}$, we have either $f^{\bar{n}} = \iota_g$ or $f^{\bar{n}} = \opn{id}_{\Sigma_{g}}$.
            The latter is equivalent to $n = \bar{n}$, and the former implies $n = 2\bar{n}$.
            The proof is concluded.
        \end{proof}
        
        \begin{cor}
        \label{cor:1}
            $n$ is even and $G$ is cyclic if and only if $n = 2 \bar{n}$.
        \end{cor}
        
        \begin{proof}
                By Lemma \ref{lem:1}, the necessity holds.
            Assume that $n$ is even and $G$ is cyclic.
                Then we have $f^{\frac{n}{2}} = \iota_g$, which implies that $G/\langle \iota_g \rangle$ is of order $\frac{n}{2}$.
                Since $\bar{f}$ is a generator of $G/\langle \iota_g \rangle$ acting on $\Sigma_g$, it follows that $\bar{n} = \displaystyle\frac{n}{2}$.
        \end{proof}
        
    We will use the following result.
        
        \begin{lem}
            \label{lem:fix3}
                For any $n \in \ZZ_{\geq 2}$, there exists no $n$-fold cyclic branched covering $p : S^{2} \to S^{2}$ which has more than two branched points such that for the branched points $\{ x_1, x_2, x_3 \}$, $p^{-1}(x_{i})$ is a one-point set for $i=1,2,3$.
                In particular, there exists no periodic diffeomorphism of $S^{2}$ which fixes more than two points except the identity.
            \end{lem}
            
            \begin{proof}
                Assume that there exists an $n$-fold branched covering $p : S^{2} \to S^{2}$ which has more than two branched points $x_{1}, x_{2}, \dots, x_{s} \in S^{2}$ such that $p^{-1}(x_{i})$ is a one-point set for $i=1,2,3$.
                By Riemann-Hurwitz formula (Proposition \ref{prop:RH}), we have
                    \[
                        -2 = n \left( -2 + 3 \left( 1 - \frac{1}{n} \right) + \sum_{i=4}^{s}\left( 1 - \frac{1}{\lambda_{i}} \right) \right),
                    \]
                where $\lambda_{i}$ is the branching index of $x_{i}$ for $i=4, \dots, s$. Then we have 
                    \[
                         1-n = \sum_{i=4}^{s}\left( 1 - \frac{1}{\lambda_{i}} \right) \geq 0,
                    \]
                which implies $n \leq 1$, and hence a contradiction.
            \end{proof}
    
    The following simple observation is fundamental. 
    
        \begin{lem}
        \label{lem:pres}
            \begin{enumerate}
                \item We have $f \left(\opn{Fix} (\iota_{g}) \right) = \opn{Fix} (\iota_{g})$, where $\opn{Fix} (\iota_{g})$ is the fixed-point set of $\iota_g$.
                \item The involution $\iota_g$ maps every $f$-orbit to an $f$-orbit of the same valency.
            \end{enumerate}
        \end{lem}
        
        \begin{proof}
                For any $x \in \operatorname{Fix}(\iota_g)$, we have $\iota_{g}(f(x)) = f(\iota_g(x))= f(x)$, which implies $f(x) \in \operatorname{Fix}(\iota_g)$.
                We can prove $f^{-1}(x) \in \operatorname{Fix}(\iota_g)$ in a similar way.
                Therefore we have the first assertion $f \left(\opn{Fix} (\iota_{g}) \right) = \opn{Fix} (\iota_{g})$.
    
                Let $x \in \Sigma_g$ be a multiple point of period $\ell < n$.
                Take a disk neighborhood $U$ of $x$ such that $f^{\ell}$ acts on $U$ by the clockwise $\frac{2\pi\mu}{\lambda}$-rotation for some $\displaystyle 1 \leq \mu < \lambda$, where $\lambda = \frac{n}{\ell}$.
                Since we have $\iota_g \circ f^{\ell} \circ (\iota_g)^{-1} = f^{\ell}$, the map $f^{\ell}$ acts on $\iota_g(U)$ by the clockwise $\frac{2\pi\mu}{\lambda}$-rotation.
                Thus, by the definition, the valency of $\{x, f(x),\dots,f^{\ell-1}(x)\}$ is equal to that of $\{\iota_g(x), f( \iota_g(x)),\dots,f^{\ell-1}(\iota_g(x))\}$.
        \end{proof}
    
    In the remainder of this section, we will classify the $G$-action up to conjugacy in the case where $G$ is cyclic.
    Let $\bar{f}$ be the diffeomorphism on $\Sigma_g / \langle \iota_g \rangle$ induced by $f$.
        Let $p$ denote the double covering $\Sigma_g \to \Sigma_g / \langle \iota_g \rangle$.
    In the sequel, we will use the following criterion.
    
        \begin{lem}
        \label{lem:brodd}
            Let $x \in \Sigma_g$ be a multiple point of $f$ and $\operatorname{Stab}_G(x)$ the stabilizer subgroup of $G$ with respect to $x$.
            If $\iota_g$ fixes $x$ and $\operatorname{Stab}_G(x)$ is of even order, then $G$ is cyclic.
        \end{lem}
        
        \begin{proof}
            Let $k$ be the period of $x$.
            By assumption, the stabilizer subgroup $\operatorname{Stab}_G(x)$ of $G$ with respect to $x$ is generated by $f^{k}$ and $\iota_{g}$.
            Since $\frac{n}{k}$ is even, the subgroup of $\operatorname{Stab}_G(x)$ generated by $f^{k}$ is of even order.
            Thus, it contains $f^{\frac{n}{2}}$, which is a $\pi$-rotation around $x$.
                Since $\iota_g$ is also a $\pi$-rotation around $x$, it follows that $\iota_g \circ (f^{\frac{n}{2}})^{-1}$ acts on an neighborhood $V$ of $x$ by the identity.
                Let $h=\iota_g \circ (f^{\frac{n}{2}})^{-1}$ and $H$ be the subgroup of $\operatorname{Diff}_+(\Sigma_g)$ generated by $h$.
                Since $\Sigma_g/\langle h \rangle$ is connected, by \cite[Proposition 1.2.5]{Au04}, there exists an open dense subset $U$ of $\Sigma$ such that the isotropy group of the $H$-action is the same at all points in $x$ in $U$.
                Since $U \cap V \neq \emptyset$, it follows that $h$ fixes a point in $U$, and hence $h$ fixes all points in $U$, which implies that $h = \operatorname{id}_{\Sigma_g}$.
                Therefore, it follows that $f^{\frac{n}{2}} = \iota_{g}$, which implies that $G$ is cyclic.
        \end{proof}
    
        \begin{rem}
        \label{rem:1}
            If $\bar{f}$ is reducible, then we can take a partition $\mathcal{C}$ preserved by $\bar{f}$ such that $\mathcal{C} \cap \operatorname{Fix}(\iota_g) = \emptyset$.
            Then $G$ preserves $p^{-1}(\mathcal{C})$.
            Therefore, if $\bar{f}$ is reducible, then $G$ is reducible.
            Hence, to classify irreducible group $G$, it is sufficient to consider the case where $\bar{f}$ is irreducible.
        \end{rem}
    
        \begin{figure}[ht]
            \centering
                \includegraphics[width=0.45\textwidth]{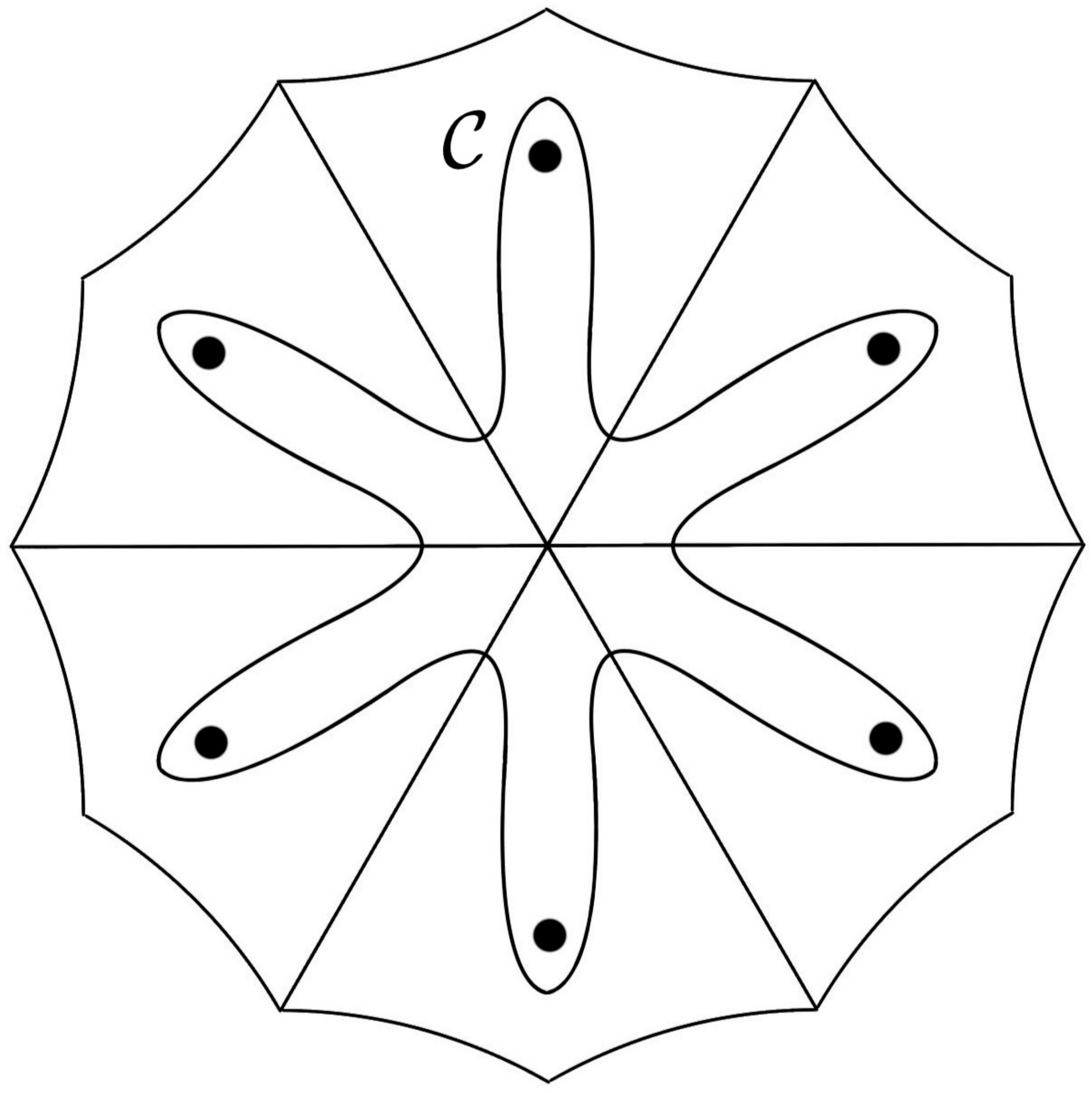}
                \caption{A simple closed curve $\mathcal{C}$}
                \label{fig:simplebr}
        \end{figure}
        
        \begin{lem}
        \label{lem:freered}
            Assume that $\bar{f}$ is conjugate a power of $h_{n,p}$.	
            If the canonical projection $p : \Sigma_g \to \Sigma_g / \langle \iota_g \rangle$ is branched at a free orbit of $\bar{f}$, then $G$ is reducible.
        \end{lem}
        
        \begin{proof}
            Assume that $p$ is branched at a point $y$, which is not a multiple point of $\bar{f}$.
            Let $\{y_0, y_1, \dots, y_{\bar{n}-1}\}$ be the $\bar{f}$-orbit of $y$.
            For $i=0,1, \dots, \bar{n}-1$, take an embedded arc $\gamma_i$ that connects $y_i$ and the barycenter $x$ of the $2n$-gon so that $\gamma_i \cap \gamma_j = \{x\}$ for $i\neq j$.
            Let $U$ be a simply-connected small open neighborhood of $\displaystyle\bigcup_{i=0}^{\bar{n}-1} \gamma_i$.
            Let $\mathcal{C}$ be the path that goes around the boundary of $U$ once, as shown in Figure \ref{fig:simplebr}.
            Then each component of the preimage $p^{-1}(\mathcal{C})$ is essential and preserved by $f$ and $\iota_g$, hence $G$ is reducible.
        \end{proof}
    
    Now, we will determine irreducible $f$ that induces $h_{4,1}$ on the torus $\Sigma_{g} / \langle \iota_g \rangle$.
    
    Since $\iota_g$ has $2g-2$ fixed points, $p$ is branched at $2g-2$ points.
    By Lemma \ref{lem:pres}, $\bar{f}$ maps the branch locus of $p$ to itself.
    By classifying nonempty $\bar{f}$-invariant sets with an even number of points in the union of the multiple orbits of $\bar{f}$, we obtain the following lemma.
    
        \begin{lem}
        \label{lem:4}
            If $\bar{f}=h_{4,1}$, then the branch locus $B_{p}$ of $p$ is one of the following:
                \begin{enumerate}
                    \item $\{ x, y \}$,\label{case:31}
                    \item $\{ z_1, z_2 \}$,\label{case:32}
                    \item $\{ x, y, z_1, z_2 \}$,\label{case:33}
                \end{enumerate}
            where $x, y, z_1, z_2 \in \Sigma_{g}/\langle \iota_g \rangle$ are as shown in Figure \ref{fig:112}.
        \end{lem}
        
        \begin{prop}
        \label{prop:1}
            If $\bar{f}=h_{4,1}$, then $G$ is conjugate to either of $\langle h_{8,1} \rangle$ or $\langle h_{8,5} \rangle$.
        \end{prop}
        
        \begin{proof}
            Let $q$ be the projection $\Sigma_g \to \Sigma_g / G$.
            We can see that $G$ must be cyclic in all three cases in Lemma \ref{lem:4}.
            Indeed, since the order of $\bar{f}$ is equal to $4$, and the period of multiple points of $\bar{f}$ is either $1$ or $2$, we can see that the branching index of every branched point of $q$ is even.
            Then, by Lemma \ref{lem:brodd}, $G$ is cyclic.
            By Corollary \ref{cor:1}, we have that $n=8$. 
    
            Let us consider the case \ref{case:31} in Lemma \ref{lem:4}.
            We can see that $q$ has three branched points of index $8,8$ and $2$, respectively.
            Then the total valency of $f$ is of the form $[\,2,8\, ; \, \frac{\theta_1}{8}+ \frac{\theta_2}{8}+ \frac{1}{2} \,]$, where $\theta_1, \theta_2 \in \{ 1,3,5,7 \}$ with $\theta_1 \leq \theta_2$.
            Here $\frac{\theta_1}{8} + \frac{\theta_2}{8} + \frac{1}{2}$ is an integer by Nielsen's condition (Proposition \ref{prop:Nielsenint}).
                Thus $\theta_1 + \theta_2$ is equal to either $4$ or $12$, and hence, we have either $(\theta_1, \theta_2) = (1,3)$ or $(5,7)$.
            The involution $f^4$ fixes two fixed points of $f$.
            In addition, since the $f$-orbit of valency $\frac{1}{2}$ consists of four points, $f^4$ fixes each point in the $f$-orbit.
            Thus, having $6$ fixed points, $f^4$ is a hyperelliptic involution.
            Therefore, despite of the value of $(\theta_1,\theta_2)$, $f^4$ is a hyperelliptic involution, which contradicts Lemma \ref{lem:1}(i).
            Thus, this case does not occur.
        
            In the case \ref{case:32} of Lemma \ref{lem:4}, we can see that $q$ has three branched points of index $4$.
            Thus, the total valency of $f$ is of the form $[\,2,8\, ; \, \frac{\theta_1}{4}+ \frac{\theta_2}{4}+ \frac{\theta_3}{4} \,]$, where $\theta_{1}, \theta_{2}, \theta_{3} \in \{1,3\}$.
                Then, since $\theta_1+\theta_2+\theta_3$ is odd, there are no $\theta_{1}, \theta_{2}, \theta_{3}$ such that $\frac{\theta_1}{4}+ \frac{\theta_2}{4}+ \frac{\theta_3}{4}$ is an integer.
            By Nielsen's condition (Proposition \ref{prop:Nielsenint}), this case does not occur.
        
            Finally, let us consider the case \ref{case:33} of Lemma \ref{lem:4}.
            We can see that $q$ has three branched points of index $8,8$ and $4$, respectively.
            Let $[\,3,8\, ; \, \frac{\theta_1}{8}+ \frac{\theta_2}{8}+ \frac{\theta_3}{4} \,]$ be the total valency of $f$, where $\theta_1, \theta_2 \in \{ 1,3,5,7 \}$ with $\theta_1 \leq \theta_2$ and $\theta_3 \in \{1,3\}$. 
            Now $\frac{\theta_1}{8} + \frac{\theta_2}{8} + \frac{\theta_3}{4}$ is an integer by Nielsen's condition (Proposition \ref{prop:Nielsenint}).
            First, consider the case where $\theta_3 = 1$.
            In this case, $\theta_1 + \theta_2$ is equal to either $6$ or $14$, and hence, we have either $(\theta_1, \theta_2,\theta_3)=(1,5,1)$, $(3,3,1)$, $(7,7,1)$.
            In the case where $\theta_3 = 3$, similarly, we have either $(\theta_1, \theta_2,\theta_3) = (1,1,3)$, $(5,5,3)$ or $(3,7,3)$.
            By Proposition \ref{prop:val}, the total valency of $h_{8,1}$ and $h_{8,5}$ are $[\, 3,8 \,;\, \frac{1}{8}+\frac{1}{8}+\frac{3}{4} \,]$ and $[\, 3,8 \,;\, \frac{1}{8}+\frac{5}{8} +\frac{1}{4} \,]$, respectively.
            Thus $f$ is conjugate to either $h_{8,1}^{k}$ for some $k \in \{1,3,5,7\}$ or $h_{8,5}^{\ell}$ for some $\ell \in \{1,3\}$, which implies that $G=\langle f \rangle$ is conjugate to either $\langle h_{8,1} \rangle$ or $\langle h_{8,5} \rangle$.
        \end{proof}
    
    Let us consider the case where $\bar{f} = h_{6,3}$.
    We obtain the following lemma by an argument similar to the case where $\bar{f} = h_{4,1}$.
    
        \begin{lem}
        \label{lem:5}
            If $\bar{f}=h_{6,3}$, then branch locus $B_p$ of $p$ is one of the following:
                \begin{enumerate}
                    \item $\{ y_1, y_2 \}$,\label{case:41}
                    \item $\{ x, z_1, z_2, z_3 \}$,\label{case:42}
                    \item $\{ x, y_1, y_2, z_1, z_2, z_3 \}$,\label{case:43}
                \end{enumerate}
            where $x, y_1, y_2, z_1, z_2, z_3 \in \Sigma/ \langle \iota_g \rangle$ are as shown in Figure \ref{fig:123}.
        \end{lem}
        
        \begin{prop}
        \label{prop:2}
            If $\bar{f}=h_{6,3}$, then $G$ is conjugate to a subgroup of either of $\langle h_{6,1}, I \rangle, \langle h_{12,2} \rangle$, or $\langle h_{12,3} \rangle$.
        \end{prop}
        
        \begin{proof}
            Let us consider the case \ref{case:42} of Lemma \ref{lem:5}.
            Since $p$ is branched at a fixed point $x$ of $\bar{f}$ and the order of $f$ is even, by Lemma \ref{lem:brodd}, $G$ is cyclic.
            Then, by Corollary \ref{cor:1}, we have $n=12$ and by the assumption on the branched locus, the branching indices of the branched points of $\bar{f}$ are $12,4$ and $3$.
            Let $[\,3,12\, ; \, \frac{\theta_1}{12}+ \frac{\theta_2}{4}+ \frac{\theta_{3}}{3} \,]$ be the total valency of $f$, where $\theta_1 \in \{ 1,5,7,11 \}, \theta_2 \in \{ 1,3 \}$ and $\theta_3 \in \{1,2 \}$. 
            By Nielsen's condition (Proposition \ref{prop:Nielsenint}), $\frac{\theta_1}{12} + \frac{\theta_2}{4} + \frac{\theta_3}{3}$ is an integer.
            By $0< \theta_1 \leq 11$, $0 < \theta_2 \leq 3$ and $0 < \theta_3 \leq 2$, we have either $\frac{\theta_1}{12} + \frac{\theta_2}{4} + \frac{\theta_3}{3} = 1$ or $2$.
            Since $\frac{\theta_2}{4} + \frac{\theta_3}{3}$ is equal to either $\frac{7}{12}$, $\frac{11}{12}$, $\frac{13}{12}$ or $\frac{17}{12}$, we can determine all $(\theta_1, \theta_2,\theta_3)$ that satisfy the condition:$(\theta_1, \theta_2,\theta_3)$ is equal to either $(5,1,1)$, $(1,1,2)$, $(11,3,1)$ or $(7,3,2)$.
        By Proposition \ref{prop:val}, the total valency of $h_{12,3}$ is $[\, 3,12 \,;\, \frac{1}{12}+\frac{1}{4}+\frac{2}{3} \,]$.
            Thus $f$ is conjugate to $h_{12,3}^{k}$ for some $k \in \{1,5,7,11\}$, which implies that $G=\langle f \rangle$ is conjugate to $\langle h_{12,3} \rangle$.
            
            Let us consider the case \ref{case:43} of Lemma \ref{lem:5}.
            Since $p$ is branched at a fixed point $x$ of $\bar{f}$ and the order of $f$ is even, by Lemma \ref{lem:brodd}, $G$ is cyclic.
            Then, by Corollary \ref{cor:1}, we have $n=12$, and by the assumption on the branched locus, the branching indices of the branched points of $\bar{f}$ are $12,6$ and $4$.
            Let $[\,3,12\, ; \, \frac{\theta_1}{12}+ \frac{\theta_2}{6}+ \frac{\theta_{3}}{4} \,]$ be the total valency of $f$, where $\theta_1 \in \{ 1,5,7,11 \}, \theta_2 \in \{ 1,5 \}$ and $\theta_3 \in \{1,3 \}$.
            By Nielsen's condition (Proposition \ref{prop:Nielsenint}), $\frac{\theta_1}{12} + \frac{\theta_2}{6} + \frac{\theta_3}{4}$ is an integer, and it is equal to either $1$ or $2$ by the same argument as the last paragraph.
            Since $\frac{\theta_2}{6} + \frac{\theta_3}{4}$ is equal to either $\frac{5}{12}$, $\frac{11}{12}$, $\frac{13}{12}$ or $\frac{19}{12}$, we can determine all $(\theta_1, \theta_2,\theta_3)$ that satisfy the condition:$(\theta_1, \theta_2,\theta_3)$ is equal to either $(7,1,1)$, $(11,5,1)$, $(1,1,3)$ or $(5,5,3)$.
            By Proposition \ref{prop:val}, the total valency of $h_{12,2}$ is $[\, 3,12 \,;\, \frac{1}{12}+\frac{1}{6}+\frac{3}{4} \,]$.
            Thus $f$ is conjugate to $h_{12,2}^{k}$ for some $k \in \{1,5,7,11\}$, which implies that $G=\langle f \rangle$ is conjugate to $\langle h_{12,2} \rangle$.
            
            Finally, let us consider the case \ref{case:41} in Lemma \ref{lem:5}.
            Since $p$ is branched at $2$ points, we have $g=2$.
            Since $\bar{f}$ fixes $x$, it follows that $p^{-1}(x)$ is either a multiple orbit of period two of $f$, or $f$ fixes $p^{-1}(x)$ pointwisely.
            In the former case, $\iota_{g} \circ f$ is a lift of $\bar{f}$ which fixes $p^{-1}(x)$ pointwisely.
            Thus, to classify the conjugacy class of $G$, we can assume that $f$ fixes $p^{-1}(x)$ pointwisely without loss of generality.
            In this case, by Wiman's result (Theorem \ref{thm:Wiman}), we have $n \leq 10$, and hence $n=6$.
            Since the valency of $\{y_{1},y_{2}\}$ as an $\bar{f}$-orbit is $\frac{1}{3}$, by a local argument, we can see that the valency of $p^{-1}(\{y_1, y_2 \})$ is either $\frac{2}{3}$ or $\frac{1}{6}$.
            But $n=6$ yields that only the former case occurs.
            Similarly, $p^{-1}(\{ z_1, z_2, z_3 \})$ is either a free orbit or the union of two multiple orbits of valency $\frac{1}{2}$.
            The Riemann-Hurwitz formula (Proposition \ref{prop:RH}) yields that only the former case occurs here.
            Indeed, if we substitute $g=2$ and $(\lambda_{1},\lambda_{2},\lambda_{3},\lambda_{4},\lambda_{5}) = (6,6,3,2,2)$ to the formula,
            then we get $g' = -\frac{1}{2}$, which contradicts to the fact that $g'$ is the genus of $\Sigma_{2}/\langle f \rangle$.
            Thus, it follows that the total valency of $f$ is $[\,2,6\,;\, \frac{1}{6}+\frac{1}{6}+\frac{2}{3} \,]$.
            Thus $f$ is conjugate to $h_{6,1}$.
            Now $f^{3} \circ \iota_{2}$ is a hyperelliptic involution which fixes all points in $p^{-1}(\{ z_1, z_2, z_3 \})$.
            By \cite[Lemma 13]{TN}, $G$ is conjugate to $\langle h_{6,1}, I \rangle$, where $I$ is a hyperelliptic involution on $\Sigma_{2}$.
        \end{proof}

    Finally, let us consider the case where $\bar{f} = h_{3,1}$.
    
    \begin{prop}
    \label{prop:3}
        If $\bar{f}=h_{3,1}$, then $G$ is conjugate to a subgroup of $\langle h_{6,1}, I \rangle$.
    \end{prop}
    
    \begin{proof}
        The multiple orbits of $h_{3,1}$ are three fixed points.
        A nonempty $\bar{f}$-invariant subset with an even number of points in the set of these three fixed points is a two-point set.
        All of them are mutually conjugate, and hence we can assume that $p$ is branched at $y$ and $z$ in Figure \ref{fig:h31}.
        Thus, the branch locus of $p$ is the same as the case \ref{case:41} in Lemma \ref{lem:5}.
        By an argument parallel to that in the final paragraph of the proof of Proposition \ref{prop:2}, we can see that $G$ is conjugate to a subgroup of $\langle h_{6,1}, I \rangle$. 
    \end{proof}

    Theorem \ref{thm:irr1} follows from Propositions \ref{prop:1}, \ref{prop:2} and \ref{prop:3}.

\section{Nonexistence of irreducible roots of $F_1$ and $F_2$}
    
    In this section, we will show the non-existence of irreducible roots of $F_1$ and $F_2$ constructed in Examples \ref{ex:redodd} and \ref{ex:redeven}.
    The argument of this section is independent of other sections.
	
	\begin{prop}
	\label{prop:redpf}
	    Neither $F_1$ nor $F_2$ is conjugate to a power of any irreducible periodic diffeomorphisms. 
	\end{prop}
	
	\begin{proof}
	    By Proposition \ref{prop:val}, the total valency of $h_{4g,2g}$ and $h_{4g,2g}^{4g-1}$ are equal to $\left[\, g, 4g \,;\, \frac{1}{4g} + \frac{2g-1}{4g} + \frac{1}{2} \,\right]$ and $\left[\, g, 4g \,;\, \frac{4g-1}{4g} + \frac{2g + 1}{4g} + \frac{1}{2} \,\right]$, respectively.
	    
	    By the construction of $F_1$ and $F_2$, the total valency of $F_1$ and $F_2$ are equal to $\left[\, 2g+1, 4g \,;\, \frac{1}{2} + \frac{1}{2} \,\right]$ and $\left[\, 2g+2, 2g+1 \,;\,\emptyset \,\right]$, respectively.
	    
	    In the case of $F_1$, by Riemann-Hurwitz formula, we have
	        \begin{align*}
	            2(2g+1) -2 &= 4g \left( 2g_1' - 2 + \frac{1}{2} + \frac{1}{2} \right),
	        \end{align*}
	    therefore, we have $g_1' = 1$, where $g_1'$ is the genus of $\Sigma_{2g+1} / \langle F_1 \rangle$.
	    Assume that there exists an irreducible periodic diffeomorphism $h$ of period $n$ such that $h^k=F_1$ for some $k \geq 2$.
            Then, we have $n = 4 \gcd(n, k)g$.
	    By Wiman's Theorem \ref{thm:Wiman}, we have $4 \gcd(n, k)g \leq 8g+6$, which implies that $\gcd(n, k)=2$.
        However, by Remark \ref{rem:1}, $h$ induces an irreducible periodic diffeomorphism of order $2$ on $T^2$.
        It contradicts Proposition \ref{prop:torusirr}.
        
        In the case of $F_2$, by Riemann-Hurwitz formula, we have
	        \begin{align*}
	            2(2g+2) -2 &= \left( 2g+1 \right) \left( 2g_2' - 2 \right),
	        \end{align*}
        therefore, we have $g_2' = 2$, where $g_2'$ is the genus of $\Sigma_{2g+2} / \langle F_2 \rangle$.
        Assume that there exists an irreducible periodic diffeomorphism $h$ of period $n$ such that $h^k=F_2$ for some $k \geq 2$.
        Then, we have $n = (2g+1)\gcd(n,k)$.
        By Wiman's Theorem \ref{thm:Wiman}, we have $(2g+1)\gcd(n,k) \leq 8g+6$, which implies that $\gcd(n,k) \leq 4$.
        However, by Remark \ref{rem:1} and the general theory of covering spaces, $h$ induces an irreducible periodic diffeomorphism of order $\gcd(n,k)$ on $\Sigma_2$.
        It contradicts Kasahara's result (Theorem \ref{thm:Kasahara}).
	\end{proof}
\appendix
\section{}
\label{app:A}

    In this appendix, we provide a more detailed explanation of $h_{2g+1,1}$ that is considered in Example \ref{ex:redeven}.
    The map $h_{4g+2,2g+1}$ is a hyperelliptic periodic diffeomorphism on $\Sigma_g$.
    Indeed, let $I = h_{4g+2,2g+1}^{2g+1}$, which is an involution that commutes with $h_{4g+2,2g+1}$.
    This $I$ fixes the $(2g+2)$-points that are the center of the $(8g+4)$-gon and the middle point of every edge of the $(8g+4)$-gon, which means that $I$ is hyperelliptic.

    \begin{lem}
    \label{lem:h2g+1}
        The diffeomorphism $h_{2g+1,1}$ is hyperelliptic and is conjugate to $h_{4g+2,2g+1}^{2g}$.
    \end{lem}

    \begin{proof}
        By Proposition \ref{prop:val}, the total valency of $h_{4g+2,2g+1}$ is equal to $\left[\, g,4g+2 \,;\, \frac{1}{4g+2}\right.$\\$\left. + \frac{g}{2g+1} + \frac{1}{2} \,\right]$.
        Since $\frac{4g+2}{\operatorname{gcd}(4g+2,4g+2)} =1$ and $1 \cdot 1 \equiv 1 \mod 4g+2$, the multiple orbit of valency $\frac{1}{4g+2}$ is a fixed point and $h_{4g+2,2g+1}$ acts on a small neighborhood of the point by $\frac{2\pi}{4g+2}$ rotation.
        Since $\frac{4g+2}{\operatorname{gcd}(2g+1,4g+2)} =2$ and $g(2g-1) \equiv 1 \mod 2g+1$, the multiple orbit of valency $\frac{g}{2g+1}$ consists of two points and $h_{4g+2,2g+1}^2$ acts on a small neighborhoods of each point by $\frac{2(2g-1)\pi}{2g+1}$ rotation.
        Since $\frac{4g+2}{\operatorname{gcd}(2,4g+2)} =2g+1$ and $1 \cdot 1 \equiv 1 \mod 2$, the multiple orbit of valency $\frac{1}{2}$ consists of $2g+1$ points and $h_{4g+2,2g+1}^{2g+1}$ acts on a small neighborhood of each points by $\frac{2\pi}{2}$ rotation.

        Thus, the $h_{4g+2,2g+1}$-orbit of the rotation angle $\frac{2\pi}{4g+2}$ is a fixed point of $h_{4g+2,2g+1}^{2g}$ of rotation angle $\frac{2g\pi}{2g+1}$, and the $h_{4g+2,2g+1}$-orbit of the rotation angle $\frac{2\pi}{2}$ is a free $h_{4g+2,2g+1}^{2g}$-orbit.
        
        On the other hand, since the $h_{4g+2,2g+1}$-orbit of the rotation angle $\frac{2(2g-1)\pi}{2g+1}$ consists of two points and $2g$ is even, the orbit splits into two fixed points of $h_{4g+2,2g+1}^{2g}$ with rotation angle $\frac{2\pi}{2g+1}$.
        Therefore, the total valency of $h_{4g+2,2g+1}^{2g}$ is equal to $\left[ \, g, 2g+1 \,;\, \frac{1}{2g+1} + \frac{1}{2g+1} + \frac{2g-1}{2g+1} \, \right]$.
        Now, by Proposition \ref{prop:val}, the total valency of $h_{2g+1,1}$ is equal to the one of $h_{4g+2,2g+1}^{2g}$, and hence, by Theorem \ref{thm:Nielsen}, we obtain the conclusion.

        Finally, by equation
        \[
            h_{2g+1,1} \circ I = h_{4g+2,2g+1}^{2g} \circ h_{4g+2,2g+1}^{2g+1} = h_{4g+2,2g+1}^{2g+1} \circ h_{4g+2,2g+1}^{2g} = I \circ h_{2g+1,1},
        \]
    the diffeomorphism $h_{2g+1,1}$ is hyperelliptic.
    \end{proof}
    
    \begin{prop}
    \label{prop:onepoint}
        The intersection of $\operatorname{Fix}(h_{2g+1,1})$ and $\operatorname{Fix}(I)$ consists of a single point.
    \end{prop}

    \begin{proof}
        By Lemma \ref{lem:h2g+1}, we can identify $h_{2g+1,1}$ with $h_{4g+2,2g+1}^{2g}$.
        Moreover, we already have $h_{4g+2,2g+1}^{2g+1}=I$.
        Therefore, $\operatorname{Fix}(h_{2g+1,1}) \cap \operatorname{Fix}(I)$ contains at least the fixed point of $h_{4g+2,2g+1}$.

        Thus, if we let the projection $\Sigma_g \to \Sigma_g/\langle h_{2g+1,1} \rangle$ be $p$, since $I$ induces an involution on $\Sigma_g / \langle h_{2g+1,1} \rangle$, which is a sphere with three cone points, $I$ exchanges two multiple orbits whose valencies are $ \frac{1}{2g+1}$ and preserves the remaining multiple orbit.
    \end{proof}

\end{document}